\documentclass[]{article}
\usepackage{amssymb,dsfont,amsmath}
\usepackage{mathtools}
\usepackage{amsthm}
\usepackage{graphicx}
\usepackage{ulem}
\usepackage{listings}
\usepackage{tikz}
\usepackage{extarrows}
\usepackage{amsthm}
\usepackage{pgfplots}
\usepackage{pgflibraryarrows}
\usepackage{pgflibrarysnakes}
\usepackage{booktabs}
\usepackage{slashbox,multirow}
\usepackage{color}
\usepackage{indentfirst}
\usepackage{amssymb}
\usepackage{bm}
\usepackage{graphicx}
\usepackage[colorlinks,
            linkcolor=red,
            anchorcolor=blue,
            citecolor=green
            ]{hyperref}
\allowdisplaybreaks
\usepackage{longtable,booktabs}
\usepackage{caption}
\usepackage[titletoc]{appendix}

\newcommand{\C}{\mathbb C}

\newcommand{\F}{\mathbb F}

\newcommand{\Q}{\mathbb Q}
\newcommand{\Z}{\mathbb Z}

\newtheorem{thm}{Theorem}[section]
\newtheorem{Lemma}{Lemma}[section]
\newtheorem{Definition}{Definition}[section]
\newtheorem{Proposition}{Proposition}[section]
\newtheorem{Remark}{Remark}[section]
\newtheorem{Example}{Example}[section]
\newtheorem{Corollary}{Corollary}[section]

\begin{document}
\title{More characterizations of generalized bent function in odd characteristic, their dual and the gray image}

\author{Libo Wang$^a$,  Baofeng Wu$^b$, Zhuojun Liu$^a$}
\date{{\small $^a$ KLMM, Academy of Mathematics and Systems Science, Chinese Academy of Sciences,
Beijing 100190,  China}\\
{\small $^b$ SKLOIS, Institute of Information Engineering, Chinese Academy of Sciences, Beijing 100093, China}
}

\maketitle
\normalsize
\begin{abstract}
  In this paper, we further investigate properties of generalized bent Boolean functions from $\Z_{p}^n$ to $\Z_{p^k}$,
  where $p$ is an odd prime and $k$ is a positive integer. For various kinds of representations,  sufficient and necessary conditions for bent-ness of such functions are given in terms of their various kinds of component functions. Furthermore, a subclass of gbent functions corresponding  to relative difference sets, which
  we call $\Z_{p^k}$-bent functions, are studied. It turns out that $\Z_{p^k}$-bent functions correspond
  to a class of vectorial bent functions, and  the property of being $\Z_{p^k}$-bent is much stronger then the standard bent-ness.  The dual and the generalized Gray  image of gbent function are also discussed. In addition, as a further generalization, we also define and give characterizations of gbent functions from $\Z_{p^l}^n$ to $\Z_{p^k}$ for a positive integer $l$ with $l<k$.

\end{abstract}

\section{Introduction}
\label{sec1}

Throughout this paper, let $\Z_{p^t}$  be the ring of integer modulo $p^t$ and $\Z_{p^t}^n$ be a free module over $\Z_{p^t}$ with rank $n$, where $p$ is a prime and $t$ and $n$ are positive integers. If $\mathbf{x}=(x_1,\dots,x_n)$ and $\mathbf{y}=(y_1,\dots,y_n)$ are two
elements of $\Z_{p^t}^n$, we define their  inner product by $\mathbf{x}\cdot\mathbf{y} = \sum_{i=1}^{n}x_iy_i$ (mod $p^t$) (without cause of confusion,
we always omit ``mod $p^t$" in the sequel). For a complex number $z=a+b\sqrt{-1}$, the absolute  value of
$z$ is $|z|=\sqrt{a^2+b^2}$ and $\bar{z}=a-b\sqrt{-1}$ is the complex conjugate of $z$, where $a$ and $b$ are real numbers.

A function from $\Z_p^n$ to $\Z_{p^k}$ is called a generalized  Boolean function on $n$ variables, the set formed by which is denoted
by $\mathcal{GB}_n^{p^k}$. Especially, $\mathcal{GB}_n^{p^k}$ denotes the set of all classical $p$-ary Boolean functions when $k=1$.
For  a function $f\in\mathcal{GB}_n^{p^k}$, the generalized Walsh-Hadamard transform, which is
a function $\mathcal{H}_f: \Z_p^n \rightarrow \C$, can be defined by
\begin{equation}
\label{1.01}
\mathcal{H}_f(\mathbf{u})= p^{-\frac{n}{2}} \sum \limits_{\mathbf{x} \in \Z_p^n} \zeta_p^{-\mathbf{u}\cdot\mathbf{x}} \zeta_{p^k}^{f(\mathbf{x})},
\end{equation}
for any $\mathbf{u} \in \Z_p^n$, where  $\zeta_p=e^{\frac{2 \pi \sqrt{-1}}{p}}$ and
$\zeta_{p^k}=e^{\frac{2 \pi \sqrt{-1}}{p^k}}$ represent the complex $p$-th and $p^k$-th primitive  roots of unity, respectively.
The inverse generalized Walsh-Hadamard transform of $f$ is
\begin{equation}
\label{1.02}
\zeta_{p^k}^{f(\mathbf{x})}= p^{-\frac{n}{2}} \sum \limits_{\mathbf{u} \in \Z_p^n} \zeta_p^{\mathbf{u}\cdot\mathbf{x}} \mathcal{H}_f(\mathbf{u}).
\end{equation}
We call the function $f$   gbent  if $|\mathcal{H}_f(\mathbf{u})| =1$ for all $\mathbf{u} \in \Z_{p}^n$.
A gbent function $f$ is regular if there exists some generalized  Boolean  function $f^\ast$ satisfying
$\mathcal{H}_f(\mathbf{u})= \zeta_{p^k}^{f^\ast(\mathbf{u})}$ for any $\mathbf{u} \in \Z_p^n$.
A gbent function $f$ is called weakly regular if there exists some generalized  Boolean  function $f^\ast$  and a complex $\alpha$
with unit magnitude satisfying $\mathcal{H}_f(\mathbf{u})=\alpha \zeta_{p^k}^{f^\ast(\mathbf{u})}$ for any $\mathbf{u} \in \Z_p^n$.
Such a function $f^\ast$ is called the dual of $f$. From the inverse generalized Walsh-Hadamard transform, it is easy to see that
the dual $f^\ast$ of a regular (weakly regular)  gbent function $f$ is also regular (weakly regular) gbent.

Currently there is a lot of research  regarding  constructions and analysis of gbent functions both in even and odd characteristic;
see for instance \cite{Hodzic,Hodzic1,liuhaiying,Martinsen1,Martinsen2,Martinsen,Schmidt,Stanica,Tang,Wang}.
In \cite{Schmidt}, Schmidt proposed the gbent  functions from $\Z_2^n$ to $\Z_4$, which can be used to constant amplitude codes
and $\Z_4$-linear codes for CDMA communications. Later, gbent  functions from $\Z_2^n$ to $\Z_8$ and $\Z_{16}$ were studied in
\cite{Stanica} and \cite{Martinsen1}, respectively.  Recently, a generalization of bent functions from $\Z_{2}^n$ to $\Z_q$, where $q \geq 2$
is any positive even integer, have attracted  much more attention.
Existence, characterizations and constructions of them were studied by several authors \cite{Hodzic,liuhaiying,Martinsen,Tang}. A generalization to the odd characteristic case were also given by the authors \cite{Wang} in a recent work, and gbent functions from  $\Z_p^n$ to $\Z_{p^k}$ for an odd prime $p$ were studied.

In this paper, we further investigate properties of gbent functions from $\Z_{p}^n$ to $\Z_{p^k}$ in the odd characteristic case. Firstly, for various kinds of representations,  sufficient and necessary conditions for gbent functions are given in terms of their various kinds of component functions. Secondly, we emphasize that a gbent function conceptually does not correspond to a bent function,
since in the definition of generalized Walsh-Hadamard transform not all characters of $\Z_p^n \times \Z_{p^k}$ are considered. Thus, in general,
a gbent function does not give rise to a relative difference set. For this reason we extend the definition and introduce the term of $\Z_{p^k}$-bent
function. We call a function $f\in \mathcal{GB}_n^{p^k}$ is $\Z_{p^k}$-bent if
\begin{equation*}
\mathcal{H}_f^{(p^k)}(a,\mathbf{u})= p^{-\frac{n}{2}} \sum\limits_{\mathbf{x} \in \Z_{p}^n} \zeta_{p^k}^{af(\mathbf{x})}  \zeta_p^{-\mathbf{u} \cdot \mathbf{x}}
\end{equation*}
has absolute value $1$ for all $\mathbf{u} \in \Z_p^n$ and all nonzero $a \in \Z_{p^k}$. The property of being $\Z_{p^k}$-bent
is much stronger then the standard concept of bent-ness. Therefore,  $\Z_{p^k}$-bent functions seem not  easy to obtain. However, we give a construction using partial spreads.

In addition,  from the result of \cite{Wang}, we know that for gbent function $f \in \mathcal{GB}_n^{p^k}$,  there exists a function $f^{\ast}: \Z_p^{n}\rightarrow \Z_{p^k}$ such that
\begin{eqnarray*}
 \mathcal{H}_f(\mathbf{u})=\left\{
\begin{array}{l}
\pm \zeta_{p^k}^{f^{\ast}(\mathbf{u})}  ~~~~~~~~~~~~$if$~~$n$~~is~~even~~or~~$n$~~is ~~odd~~and~~p \equiv 1(mod~4),\\
\pm \sqrt{-1} \zeta_{p^k}^{f^{\ast}(\mathbf{u})} ~~~~ $if$~~ $n$~~is ~~odd~~and~~p \equiv 3(mod~4).
\end{array}
\right.
\end{eqnarray*}
Similar as in \cite{Cesmelioglu1},  we also call this $f^{\ast}$ the dual of the gbent function $f$. Note that this definition is much more formal, since $f$ may not be (weakly) regular (this is because $+$ and $-$ may appear in $\mathcal{H}_f(\mathbf{u})$ for different $\mathbf{u}$). We emphasis that when $f$ is non-(weakly) regular, $f^{\ast}$ may not be a gbent function.  Anyway, we can explicitly determine $f^*$ in terms of the dual of component functions of $f$.

Furthermore, every function in $ \mathcal{GB}_n^{p^k}$ have a so-called generalized Gray image, which is a function in $ \mathcal{GB}_{n+k-1}^{p}$. We show that  the generalized Gray image of a gbent function in $\mathcal{GB}_n^{p^k}$ is a
 $(k-1)$-plateaued function, where $(k-1)$-plateaued functions in odd characteristic are given in Definition \ref{Definition5.2.1}.

 At last,  we also further generalize our study of gbent functions in $ \mathcal{GB}_n^{p^k}$ to gbent functions from $\Z_{p^l}^n$ to $\Z_{p^k}$ for a positive integer $l$ with $l<k$, giving their definition  and  characterizations.

The rest of this paper is organized as follows, In Section \ref{sec2} we give some preliminary results which will be used later.
Some different descriptions for a function in $\mathcal{GB}_n^{p^k}$ to be gbent is given in Section \ref{sec3.1}.
$\Z_{p^k}$-bent functions and their relationships  to
relative difference sets in $\Z_p^n \times \Z_{p^k}$ are introduced  in Section \ref{sec4}. In Section \ref{sec5} we specify
the dual and generalized Gray map of gbent functions in $\mathcal{GB}_n^{p^k}$. Finally, gbent functions from $\Z_{p^l}^n$ to $\Z_{p^k}$ are introduced and characterized
in Section \ref{sec3.2}. Conclusions are given in Section \ref{secconclu}.

\section{Preliminaries}
\label{sec2}

In this section we will give some results on cyclotomic fields, which will be used in the following sections.
Firstly, we state some basic facts  on the cyclotomic fields $K=\Q(\zeta_{p^k})$, which can be found in any book on algebraic number theory,
for example \cite{Washington}.

Let $\mathcal{O}_K$ be the ring of integers of $K=\Q(\zeta_{p^k})$. It is well known that $\mathcal{O}_K=\Z[\zeta_{p^k}]$. Any nonzero ideal $A$
of $\mathcal{O}_K$ can be uniquely (up to the order) expressed as
\[A=P_1^{a_1} \cdots P_s^{a_s},\]
where $P_1, \cdots, P_s$ are distinct prime ideals of $\mathcal{O}_K$ and $a_i \geq 1$, for $1\leq i \leq s$. In other words, the set $S(K)$ of
all the nonzero ideals of $\mathcal{O}_K$ is a free multiplicative  communicative semigroup with a basis $B(K)$, the set of all nonzero  prime ideals
of $\mathcal{O}_K$. Such semigroup $S(K)$ can be extended to the commutative group $I(K)$, called the group of fractional ideals of $K$. Each element of $I(K)$,
called a fractional ideals, has the form $AB^{-1}$, where $A,B$ are ideals of $\mathcal{O}_K$. For each $\alpha \in K^{\ast}=K \setminus \{0\}$, $\alpha \mathcal{O}_K$
is a fractional ideals, called a principle fractional ideals, and we have $(\alpha \mathcal{O}_K)(\beta \mathcal{O}_K)=\alpha \beta \mathcal{O}_K$,
$(\alpha \mathcal{O}_K)^{-1}=(\alpha^{-1}) \mathcal{O}_K$. Therefore, the set $P(K)$ of all principle fractional ideals is a subgroup of $I(K)$.
Some results on $K$ are given in the following lemmas.

\begin{Lemma}
\label{Lemma2.1}
Let $k\geq 2$ and $K=\Q(\zeta_{p^k})$, then

$(i)$ The field extension $K/\Q$ is Galois of degree $(p-1)p^{k-1}$ and the Galois group  $Gal(K/ \Q)=\{ \sigma_j | ~ j \in \Z, (j,p)=1 \}$,
where the automorphism $\sigma_j$ of $K$ is defined by $\zeta_{p^k} \mapsto \zeta_{p^k}^j$.

$(ii)$  The ring of integers in $K$ is $\mathcal{O}_K = \Z [\zeta_{p^k}]$ and $\{\zeta_{p^k}^j|~0 \leq j \leq (p-1)p^{k-1}-1\}$ is an integral basis of $\mathcal{O}_K$.
The group of roots of unity in $\mathcal{O}_K$ is $W_K=\{\zeta_{{2p}^k}^j|~0 \leq j \leq 2 p^k-1\}$.

$(iii)$  The principle ideal $(1-\zeta_{p^k}) \mathcal{O}_K$ is a prime ideal of $\mathcal{O}_K$ and the rational prime $p$ is totally ramified in $\mathcal{O}_K$,
i.e., $p \mathcal{O}_K=( (1-\zeta_{p^k})\mathcal{O}_K )^{(p-1)p^{k-1}}$.
\end{Lemma}

\begin{Lemma}
[see \cite{Kumar}]
\label{Lemma2.2}
For a positive integer $q$,

$(i)$  If $q\equiv 0,1~~(mod~4)$, then $\sqrt{q} \in \Q(\zeta_q)$,

$(ii)$  If $q\equiv 2,3~~(mod~4)$, then $\sqrt{q} \in \Q(\zeta_{4q}) \backslash \Q(\zeta_{2q})$.
\end{Lemma}

\begin{Lemma}
[see \cite{Wang}]
\label{Lemma2.3}
For gbent function $f \in \mathcal{GB}_n^{p^k}$,  there exists a function $f^{\ast}: \Z_p^{n}\rightarrow \Z_{p^k}$ such that
\begin{eqnarray*}
 \mathcal{H}_f(\mathbf{u})=\left\{
\begin{array}{l}
\pm \zeta_{p^k}^{f^{\ast}(\mathbf{u})}  ~~~~~~~~~~~~$if$~~$n$~~is~~even~~or~~$n$~~is ~~odd~~and~~p \equiv 1(mod~4),\\
\pm \sqrt{-1} \zeta_{p^k}^{f^{\ast}(\mathbf{u})} ~~~~ $if$~~ $n$~~is ~~odd~~and~~p \equiv 3(mod~4).
\end{array}
\right.
\end{eqnarray*}
\end{Lemma}

\begin{Lemma}
\label{Lemma2.4}
Let $k$ is a positive integer and $a\in \Z_{p^t}$.  Then
\begin{eqnarray*}
\begin{split}
\zeta_{p^k}^a=\frac{1}{p^t} \sum \limits_{i\in \Z_{p^t}} \left(\sum \limits_{j\in \Z_{p^t}} \zeta_{p^t}^{(a-i)j}\right) \zeta_{p^k}^i
\end{split}
\end{eqnarray*}
\end{Lemma}
\begin{proof}
Let 
\begin{gather*}
\mathcal{V}_{p^t}(\zeta_{p^t})=
\begin{pmatrix}
1 &  1 & \cdots & 1 \\
1 &  \zeta_{p^t} & \cdots & \zeta_{p^t}^{p^t-1}\\
\vdots & \vdots & \ddots & \vdots\\
1 &¡¡\zeta_{p^t}^{p^t-1}  & \cdots & \zeta_{p^t}^{(p^t-1)(p^t-1)}
\end{pmatrix}
\end{gather*}
and
\begin{gather*}
\mathcal{V}_{p^t}(\zeta_{p^t}^{-1})=
\begin{pmatrix}
1 &  1 & \cdots & 1 \\
1 &  \zeta_{p^t}^{-1} & \cdots & \zeta_{p^t}^{-(p^t-1)}\\
\vdots & \vdots & \ddots & \vdots\\
1 &¡¡\zeta_{p^t}^{-(p^t-1)}  & \cdots & \zeta_{p^t}^{-(p^t-1)(p^t-1)}
\end{pmatrix}.
\end{gather*}
In fact, we know that  $\mathcal{V}_{p^t}(\zeta_{p^t})$ is a generalized Hadamard matrix, i.e., $\mathcal{V}_{p^t}(\zeta_{p^t}) (\mathcal{V}_{p^t}(\overline{\zeta_{p^t}}))^{\mathrm{T}}= p^t \mathrm{I}_{p^t}$,
and $(\mathcal{V}_{p^t}(\overline{\zeta_{p^t}}))^{\mathrm{T}}= \mathcal{V}_{p^t}(\zeta_{p^t}^{-1})$, therefore, we have
\begin{equation}
\label{2.01}
\mathcal{V}_{p^t}(\zeta_{p^t})\mathcal{V}_{p^t}(\zeta_{p^t}^{-1}) = p^t \mathrm{I}_{p^t},
\end{equation}
where $\mathrm{I}_{p^t}$ stands for the identity matrix of size $p^t$. Define now a collection of maps from $\C$ to itself by setting
\begin{equation*}
\left(
\begin{array}{c}
h_0(z)\\
h_1(z)\\
\vdots\\
h_{p^t-1}(z)
\end{array}
\right)  =
\mathcal{V}_{p^t}(\zeta_{p^t}^{-1})
\left(
\begin{array}{cccc}
1\\
z\\
\vdots\\
z^{p^t-1},
\end{array}
\right)
\end{equation*}
or equivalently, for any $j \in \Z_{p^t}$,
\begin{equation}
\label{2.02}
h_j(z)=\sum \limits_{i \in \Z_{p^t}} \zeta_{p^t}^{-ji}z^i.
\end{equation}
Furthermore, according to (\ref{2.01}), we have, for any $z \in \C$,\
\begin{equation*}
\left(
\begin{array}{cccc}
1\\
z\\
\vdots\\
z^{p^t-1}
\end{array}
\right)
 =\frac{1}{p^t}\mathcal{V}_{p^t}(\zeta_{p^t})
\left(
\begin{array}{c}
h_0(z)\\
h_1(z)\\
\vdots\\
h_{p^t-1}(z)
\end{array}
\right)
\end{equation*}
that is, for any $a \in \Z_{p^t}$,
\begin{equation}
\label{2.03}
z^a=\frac{1}{p^t} \sum \limits_{j \in \Z_{p^t}} \zeta_{p^t}^{aj}h_j(z).
\end{equation}
Then plugging (\ref{2.02}) into (\ref{2.03}),  we have
\begin{equation}
\label{2.04}
z^a=\frac{1}{p^t} \sum \limits_{i \in \Z_{p^t}} \left(\sum \limits_{j \in \Z_{p^t}}  \zeta_{p^t}^{(a-i)j} \right) z^i.
\end{equation}
Setting $z= \zeta_{p^k}$ and plugging it into (\ref{2.04}),  we get
\begin{equation*}
\zeta_{p^k}^a=\frac{1}{p^t} \sum \limits_{i \in \Z_{p^t}} \left(\sum \limits_{j \in \Z_{p^t}}  \zeta_{p^t}^{(a-i)j} \right) \zeta_{p^k}^i.
\end{equation*}
\end{proof}

\begin{Remark}
\label{Remark2.1}
Lemma  \ref{Lemma2.4} generalizes \cite[Lemma 2.6]{Wang}.
\end{Remark}

\begin{Lemma}
\label{Lemma2.5}
Let $k$, $t$ be positive integers and $k \geq t$. Then we have

$(i)$. $\{1, \zeta_{p^k},\zeta_{p^k}^2,\ldots,\zeta_{p^k}^{p^{k-t}-1}\}$ is a basis of  $\mathbb{Q}( \zeta_{p^k})$
over $\mathbb{Q}(\zeta_{p^t})$;

$(ii)$. $\{1, \zeta_{p^k},\zeta_{p^k}^2,\ldots,\zeta_{p^k}^{p^{k-t}-1}\}$ is a basis of  $\mathbb{Q}(\zeta_{p^k},\sqrt{-1})$
over $\mathbb{Q}(\zeta_{p^t},\sqrt{-1})$.
\end{Lemma}
\begin{proof}
(i). Since $[\Q({\zeta_{p^k}}): \Q({\zeta_{p^t}})] = \frac{[\Q({\zeta_{p^k}}): \Q]} {[\Q({\zeta_{p^t}}): \Q]} = \frac{(p-1)p^{k-1}}{(p-1)p^{t-1}} = p^{k-t}$,
we need only to prove $\{1, \zeta_{p^k},\zeta_{p^k}^2,\ldots,\zeta_{p^k}^{p^{k-t}-1}\}$
is linear independently over $\mathbb{Q}(\zeta_{p^t})$.

Suppose  there exists $a_i= \sum_{j=0}^{(p-1)p^{t-1}-1}a_{ij}\zeta_{p^t}^{j} \in \mathbb{Q}(\zeta_{p^t})$, $ 0\leq i \leq p^{k-t}-1$,
such that
\[\sum \limits_{i=0}^{p^{k-t}-1} a_i \zeta_{p^k}^{i}=0,\]
i.e.,
\[\sum \limits_{i=0}^{p^{k-t}-1} \sum \limits_{j=0}^{(p-1)p^{t-1}-1}a_{ij}\zeta_{p^t}^{j} \zeta_{p^k}^{i}=
\sum \limits_{i=0}^{p^{k-t}-1} \sum \limits_{j=0}^{(p-1)p^{t-1}-1}a_{ij} \zeta_{p^k}^{i+jp^{k-t}}=0.\]
It is well known that $\{1, \zeta_{p^k},\zeta_{p^k}^2,\ldots,\zeta_{p^k}^{(p-1)p^{k-1}-1}\}$ is a basis of $\mathbb{Q}(\zeta_{p^k})$ over $\mathbb{Q}$, thus
$a_{ij}=0$ for all $0 \leq i \leq p^{k-t}-1$ and  $0\leq j \leq (p-1)p^{t-1}-1$, i.e., $a_i=0$ for $0 \leq i \leq p^{k-t}-1$. 

(ii). Suppose  there exists $a_i= \sum_{j=0}^{(p-1)p^{t-1}-1}a_{ij}\zeta_{p^t}^{j} \in \mathbb{Q}(\zeta_{p^t},\sqrt{-1})$, $ 0\leq i \leq p^{k-t}-1$,
and $a_{ij}=b_{ij}+c_{ij}\sqrt{-1}$,
such that
\[\sum \limits_{i=0}^{p^{k-t}-1} a_i \zeta_{p^k}^{i}=0,\]
i.e.,
\begin{eqnarray*}
\sum \limits_{i=0}^{p^{k-t}-1} \sum \limits_{j=0}^{(p-1)p^{t-1}-1}a_{ij}\zeta_{p^t}^{j} \zeta_{p^k}^{i}&=&
\sum \limits_{i=0}^{p^{k-t}-1} \sum \limits_{j=0}^{(p-1)p^{t-1}-1}a_{ij} \zeta_{p^k}^{i+jp^{k-t}}\\
&=&\sum \limits_{i=0}^{p^{k-t}-1} \sum \limits_{j=0}^{(p-1)p^{t-1}-1}b_{ij} \zeta_{p^k}^{i+jp^{k-t}}
+\sqrt{-1}\sum \limits_{i=0}^{p^{k-t}-1} \sum \limits_{j=0}^{(p-1)p^{t-1}-1}c_{ij} \zeta_{p^k}^{i+jp^{k-t}}\\
&=&0.
\end{eqnarray*}
If $\sum \limits_{i=0}^{p^{k-t}-1} \sum \limits_{j=0}^{(p-1)p^{t-1}-1}b_{ij} \zeta_{p^k}^{i+jp^{k-t}} \neq 0$, then $\sqrt{-1} \in \mathbb{Q}(\zeta_{p^k})$, which is a contradiction.
Therefore, $$\sum \limits_{i=0}^{p^{k-t}-1} \sum \limits_{j=0}^{(p-1)p^{t-1}-1}b_{ij} \zeta_{p^k}^{i+jp^{k-t}} =\sum \limits_{i=0}^{p^{k-t}-1} \sum \limits_{j=0}^{(p-1)p^{t-1}-1}c_{ij} \zeta_{p^k}^{i+jp^{k-t}}= 0.$$
Similarly as in (i), $b_{ij}$ and $c_{ij}$ equal to 0 for all $0 \leq i \leq p^{k-t}-1$ and  $0\leq j \leq (p-1)p^{t-1}-1$, i.e.,  $a_i=0$ for $0 \leq i \leq p^{k-t}-1$.
\end{proof}

\begin{Remark}
\label{Remark2.2}
 Lemma \ref{Lemma2.5} generalizes \cite[Lemma 2.7]{Wang}.
\end{Remark}

\begin{Lemma}
\label{Lemma2.6}
Let  $\gamma_{\textbf{a}}=\sum_{\textbf{v}\in \Z_{p^t}^{l-1}} \zeta_{p^t}^{-\textbf{a} \cdot \textbf{v}} \zeta_{p^k}^{\sum_{j=1}^{l-1}p^{(j-1)t}v_j}$,
where $\textbf{a} \in \Z_{p^t}^{l-1}$ and $\textbf{v}=(v_1,v_2,\ldots, v_{l-1}) \in \Z_{p^t}^{l-1}$. Then
\[\zeta_{p^k}^{e}=\frac{1}{p^{t(l-1)}} \sum \limits_{\textbf{a}\in \Z_{p^t}^{l-1}} \zeta_{p^t}^{\textbf{a}\cdot \textbf{u}} \gamma_{\textbf{a}}, \]
where $e=\sum_{j=1}^{l-1} u_j p^{(j-1)t}$ and $\textbf{u}=(u_1,u_2,\ldots, u_{l-1}) \in \Z_{p^t}^{l-1}$.
\end{Lemma}
\begin{proof}
\begin{eqnarray*}
\frac{1}{p^{t(l-1)}} \sum \limits_{\textbf{a}\in \Z_{p^t}^{l-1}} \zeta_{p^t}^{\textbf{a}\cdot \textbf{u}} \gamma_{\textbf{a}}&=&\frac{1}{p^{t(l-1)}} \sum \limits_{\textbf{a}\in \Z_{p^t}^{l-1}} \zeta_{p^t}^{\textbf{a}\cdot \textbf{u}}
\sum_{\textbf{v}\in \Z_{p^t}^{l-1}} \zeta_{p^t}^{-\textbf{a} \cdot \textbf{v}} \zeta_{p^k}^{\sum_{j=1}^{l-1}p^{(j-1)t}v_j}\\
&=&\frac{1}{p^{t(l-1)}}\sum_{\textbf{v}\in \Z_{p^t}^{l-1}} \zeta_{p^k}^{\sum_{j=1}^{l-1}p^{(j-1)t}v_j}
\sum \limits_{\textbf{a}\in \Z_{p^t}^{l-1}} \zeta_{p^t}^{\textbf{a}\cdot (\textbf{u}-\textbf{v})}\\
&=&\frac{1}{p^{t(l-1)}}  \cdot p^{t(l-1)}  \zeta_{p^k}^{\sum_{j=1}^{l-1}p^{(j-1)t}u_j} = \zeta_{p^k}^e.
\end{eqnarray*}
\end{proof}

\begin{Remark}
\label{Remark2.3}
 Lemma \ref{Lemma2.6} coincides with  \cite[Lemma 2.8]{Wang}, if we set $l=k$ when $t=1$, $l=k$.
\end{Remark}

\begin{Remark}
\label{Remark2.4}
Note that when $k=lt$, $\{1, \zeta_{p^k},\zeta_{p^k}^2,\ldots,\zeta_{p^k}^{p^{k-t}-1}\}$ is a basis of  $\mathbb{Q}( \zeta_{p^k})$
over $\mathbb{Q}(\zeta_{p^t})$, and $\zeta_{p^k}^e, 0 \leq e \leq p^{k-t}-1$, can be expressed by $\gamma_{\textbf{a}}, \textbf{a}\in \Z_{p^t}^{l-1}$,
where the coefficients can form a non-singular matrix over  $\mathbb{Q}(\zeta_{p^t})$, by Lemma \ref{Lemma2.6}.
So $\{\gamma_{\textbf{a}} |¡¡\textbf{a}\in \Z_{p^t}^{l-1}¡¡\}$ is also a basis of $\mathbb{Q}( \zeta_{p^k})$ over $\mathbb{Q}(\zeta_{p^t})$.
Similarly, when $k=lt$,  $\{\gamma_{\textbf{a}} |¡¡\textbf{a}\in \Z_{p^t}^{l-1}¡¡\}$ is also a basis of $\mathbb{Q}(\zeta_{p^k}, \sqrt{-1})$
over $\mathbb{Q}(\zeta_{p^t},\sqrt{-1})$.
\end{Remark}

\begin{Lemma}
\label{Lemma2.7}
Let $\textbf{a}=(a_1,\cdots, a_{l-1}) \in \Z_{p^t}^{l-1}$ and $\gamma_{\textbf{a}}=\sum_{\textbf{v}\in \Z_{p^t}^{l-1}} \zeta_{p^t}^{-\textbf{a} \cdot \textbf{v}} \zeta_{p^k}^{\sum_{i=1}^{l-1}p^{(i-1)t}v_i}$.
Then
\[\gamma_{\textbf{a}}= \prod \limits_{i=1}^{l-1}\left( \sum \limits_{v_i \in \Z_{p^t}} \zeta_{p^t}^{-v_ia_i} \zeta_{p^{k}}^{p^{(i-1)t}v_i}\right).\]
\end{Lemma}
\begin{proof}
Obvious.
\end{proof}

\begin{Remark}
\label{Remark2.5}
Lemma \ref{Lemma2.7} is equivalent to   \cite[Lemma 2.9]{Wang}.
\end{Remark}

\section{More characterizations for gbent functions}
\label{sec3.1}

In this section, we give more characterizations for gbent functions
in $\mathcal{GB}_n^{p^k}$,  extending the work in \cite{Wang}.

\begin{Lemma}
\label{Lemma3.1.01}
Let $k \geq 2t$ and $f: \Z_{p}^n \rightarrow \Z_{p^k}$ be defined by
\[f(\mathbf{x})=\sum \limits_{i=1}^{k-1}p^{i-1}a_i(\mathbf{x})=g(\mathbf{x})+p^th(\mathbf{x})\]
for some $p$-ary Boolean functions $a_i : \Z_p^n \rightarrow \Z_p, 1 \leq i \leq k$, and
\begin{eqnarray*}
\begin{split}
g(\mathbf{x})=\sum\limits_{i=1}^t p^{i-1}a_i(\mathbf{x}) \in \mathcal{GB}_n^{p^t},~~
h(\mathbf{x})=\sum \limits_{i=1}^{k-t}p^{i-1}a_{t+i}(\mathbf{x}) \in \mathcal{GB}_n^{p^{k-t}}.
\end{split}
\end{eqnarray*}
Then for every  $\mathbf{u} \in \Z_p^n$, $$\mathcal{H}_f(\mathbf{u})=\frac{1}{p^t} \sum \limits_{c \in \Z_{p^t}} \mathcal{H}_{h+cp^{k-2t}g}(\mathbf{u}) \gamma_c,$$
where $\gamma_c=\sum \limits_{d \in \Z_{p^t}} \zeta_{p^t}^{-cd}\zeta_{p^k}^d$.
\end{Lemma}

\begin{proof}
By the definition of generalized Walsh-Hadamard transform, we have
\begin{eqnarray*}
p^{\frac{n}{2}} \mathcal{H}_f(\mathbf{u})&=& \sum \limits_{\mathbf{x} \in \Z_p^n} \zeta_{p^k}^{f(\mathbf{x})} \zeta_p^{-\mathbf{u}\cdot\mathbf{x}}
=\sum \limits_{\mathbf{x} \in \Z_p^n} \zeta_{p^k}^{g(\mathbf{x})}  \zeta_{p^{k-t}}^{h(\mathbf{x})} \zeta_p^{-\mathbf{u}\cdot\mathbf{x}}\\
&=&\sum \limits_{\mathbf{x} \in \Z_p^n} \zeta_{p^{k-t}}^{h(\mathbf{x})} \zeta_p^{-\mathbf{u}\cdot\mathbf{x}} \left(\frac{1}{p^t} \sum \limits_{d \in \Z_{p^t}}
\left(\sum \limits_{c \in \Z_{p^t}} \zeta_{p^t}^{(g(\mathbf{x})-d)c} \right)\zeta_{p^k}^d\right) \\
&=&\frac{1}{p^t} \sum \limits_{\mathbf{x} \in \Z_p^n} \zeta_{p^{k-t}}^{h(\mathbf{x})} \zeta_p^{-\mathbf{u}\cdot\mathbf{x}} \left(\sum \limits_{c \in \Z_{p^t}}
\left(\sum \limits_{d \in \Z_{p^t}} \zeta_{p^t}^{-cd}\zeta_{p^k}^d \right) \zeta_{p^t}^{cg(\mathbf{x})}\right) \\
&=&\frac{1}{p^t}\sum \limits_{c \in \Z_{p^t}} \sum \limits_{\mathbf{x} \in \Z_p^n} \zeta_{p^{k-t}}^{h(\mathbf{x})} \zeta_p^{-\mathbf{u}\cdot\mathbf{x}} \zeta_{p^t}^{cg(\mathbf{x})}
\left(\sum \limits_{d \in \Z_{p^t}} \zeta_{p^t}^{-cd}\zeta_{p^k}^d \right) \\
&=&\frac{1}{p^t}\sum \limits_{c \in \Z_{p^t}} \sum \limits_{\mathbf{x} \in \Z_p^n} \zeta_{p^{k-t}}^{h(\mathbf{x})+cp^{k-2t}g(\mathbf{x})} \zeta_p^{-\mathbf{u}\cdot\mathbf{x}}
\left(\sum \limits_{d \in \Z_{p^t}} \zeta_{p^t}^{-cd}\zeta_{p^k}^d \right) \\
&=&\frac{p^{\frac{n}{2}}}{p^t}\sum \limits_{c \in \Z_{p^t}} \mathcal{H}_{h+cp^{k-2t}g} (\mathbf{u})\gamma_c.
\end{eqnarray*}
This completes the proof.
\end{proof}

\begin{thm}
\label{thm3.1.01}
Let $k \geq 2t$ and $f: \Z_{p}^n \rightarrow \Z_{p^k}$ be defined by
\[f(\mathbf{x})=\sum \limits_{i=1}^{k-1}p^{i-1}a_i(\mathbf{x})=g(\mathbf{x})+p^th(\mathbf{x}),\]
where $a_i$, $g$ and $h$ are defined as in Lemma \ref{Lemma3.1.01}. Then
$f$ is gbent if and only if for any $\mathbf{u} \in \Z_p^n$ and $c \in \Z_{p^t}$, there exist some $d \in \Z_{p^t}$ and $j \in \Z_{p^{k-t}}$ such that
\begin{eqnarray*}
 \mathcal{H}_{h+cp^{k-2t}g}(\mathbf{u})=\left\{
\begin{array}{l}
\pm \zeta_{p^{k-t}}^{j} \zeta_{p^t}^{cd} ~~~~~~~~~~~~$if$~~$n$~~is~~even~~or~~$n$~~is ~~odd~~and~~p \equiv 1(mod~4),\\
\pm \sqrt{-1} \zeta_{p^{k-t}}^{j} \zeta_{p^t}^{cd} ~~~~ $if$~~ $n$~~is ~~odd~~and~~p \equiv 3(mod~4),
\end{array}
\right.
\end{eqnarray*}
where $h+cp^{k-2t}g \in \mathcal{GB}_n^{p^{k-t}}$ for every $c \in \Z_{p^t}$, and  $j$ and $d$ only depend on $\mathbf{u}$ and $f$.
\end{thm}

\begin{proof}
If $n$ is even or $n$ is odd and $p\equiv 1~(mod~4)$, then $\mathcal{H}_f{(\textbf{u})}= \pm \zeta_{p^k}^i$, for some $0 \leq i \leq p^k-1$,
by Lemma \ref{Lemma2.3}.
Hence, $\mathcal{H}_f{(\textbf{u})}$ can be expressed as $\mathcal{H}_f{(\textbf{u})}= \pm \zeta_{p^{k-t}}^j \zeta_{p^k}^{i-jp^t}$,
where $0 \leq j \leq p^{k-t}-1$. Let $d=i-jp^{t}$, $0 \leq d \leq p^{t}-1$.
According to Lemma \ref{Lemma3.1.01}, we have
\begin{eqnarray}
\label{3.1.01}
\mathcal{H}_f{(\textbf{u})}&=& \frac{1}{p^t} \sum \limits_{c \in \Z_{p^t}} \mathcal{H}_{h+cp^{k-2t}g}(\mathbf{u}) \gamma_c,
\end{eqnarray}
where $\gamma_c=\sum \limits_{d \in \Z_{p^t}} \zeta_{p^t}^{-cd}\zeta_{p^k}^d$. By  Lemma \ref{Lemma2.6},  we have
\[\zeta_{p^k}^{d}=\frac{1}{p^{t}} \sum \limits_{c\in \Z_{p^t}} \zeta_{p^t}^{cd} \gamma_{c}.\]
Therefore,
\begin{eqnarray}
\label{3.1.02}
\begin{split}
\mathcal{H}_f{(\textbf{u})}&= \pm \zeta_{p^{k-t}}^j \zeta_{p^k}^{i-jp^t}=\pm \zeta_{p^{k-t}}^j \zeta_{p^k}^{d}\\
&= \pm \frac{1}{p^{t}}\zeta_{p^{k-t}}^j \sum \limits_{c\in \Z_{p^t}} \zeta_{p^t}^{c d} \gamma_{c}.
\end{split}
\end{eqnarray}
Combining  (\ref{3.1.01}) with (\ref{3.1.02}),  we can get
\begin{eqnarray}
\label{3.1.03}
\begin{split}
\mathcal{H}_f{(\textbf{u})}&=\frac{1}{p^t} \sum \limits_{c \in \Z_{p^t}} \mathcal{H}_{h+cp^{k-2t}g}(\mathbf{u}) \gamma_c\\
&= \pm \frac{1}{p^{t}}\zeta_{p^{k-t}}^j \sum \limits_{c\in \Z_{p^t}} \zeta_{p^t}^{c\cdot d} \gamma_{c}.
\end{split}
\end{eqnarray}
Since $h+cp^{k-2t}g \in \mathcal{GB}_n^{p^{k-t}}$, it is absolutely that $\mathcal{H}_{h+cp^{k-2t}g}(\mathbf{u}) \in \Q(\zeta_{p^{k-t}})$
in the case of $n$ is even or $n$  is odd and $p\equiv 1~(mod~4)$. We note that $\{1,\zeta_{p^k},\cdots, \zeta_{p^k}^{p^t-1}\}$
is a basis of $\Q(\zeta_{p^k})$ over $\Q(\zeta_{p^{k-t}})$ by  Lemma \ref{Lemma2.5} (i). So, $\{\gamma_c|c\in \Z_{p^t}\}$
is also a basis of $\Q(\zeta_{p^k})$ over $\Q(\zeta_{p^{k-t}})$ by  Remark \ref{Remark2.4}. Further, by  $k\geq 2t$,  we have $\Q(\zeta_{p^t})
\subseteq \Q(\zeta_{p^{k-t}})$.

Based on the above  discussions and (\ref{3.1.03}), we can get
\[\mathcal{H}_{h+cp^{k-2t}g}(\mathbf{u})= \pm \zeta_{p^{k-t}}^j  \zeta_{p^t}^{c d}.\]

If $n$ is odd and $p\equiv 3~(mod~4)$, then $\mathcal{H}_f{(\textbf{u})}= \pm \sqrt{-1}\zeta_{p^k}^i$ for some $0 \leq i \leq p^k-1$
by Lemma \ref{Lemma2.3}. Note that in this case, by  Lemma \ref{Lemma2.2} (ii), we have
\[\mathcal{H}_{h+cp^{k-2t}g}(\mathbf{u}) \in \Q(\zeta_{4p^{k-t}}) \setminus \Q(\zeta_{2p^{k-t}})  \subseteq \Q(\zeta_{p^{k-t}},\sqrt{-1}).\]
Similar as before, we can get
\[\mathcal{H}_{h+cp^{k-2t}g}(\mathbf{u})=\pm \sqrt{-1} \zeta_{p^{k-t}}^{j} \zeta_{p^t}^{cd}.\]
\end{proof}

\begin{thm}
\label{thm3.1.02}
Let $f \in \mathcal{GB}_n^{p^k}$ and $f(\mathbf{x})=\sum\limits_{i=1}^{k}p^{i-1}a_i(\mathbf{x})$,
for some $p$-ary Boolean functions $a_i : \Z_p^n \rightarrow \Z_p, 1 \leq i \leq k$.
Let $\mathbf{c}=(c_1,\dots,c_{s-1}) \in \Z_{p^t}^{s-1}$, and denote 
\begin{eqnarray*}
g_{\mathbf{c}}(\mathbf{x})=\sum\limits_{i=1}^{k-ts}p^{i-1}a_{t(s-1)+i}(\mathbf{x})+
p^{k-ts}\left(\sum \limits_{i=1}^{s-1}c_i\left(\sum\limits_{j=1}^tp^{j-1}a_{(i-1)t+j}(\mathbf{x})\right)
+\sum\limits_{j=1}^t p^{j-1}a_{k-t+j}(\mathbf{x})\right).
\end{eqnarray*}
Then $f$ is  gbent  if and only if $g_{\mathbf{c}}$ is gbent in $\mathcal{GB}_n^{p^{k-(s-1)t}}$  for any $s$ satisfying  $st \leq k$, and
there exist  some $\mathbf{d} \in \Z_{p^t}^{s-1}$ and $j \in \Z_{p^{k-t}}$ such that
\begin{eqnarray*}
\mathcal{H}_{g_{\mathbf{c}}}(\mathbf{u})=\left\{
\begin{array}{l}
\pm \zeta_{p^{k-(s-1)t}}^{j} \zeta_{p^t}^{\mathbf{c} \cdot \mathbf{d}} ~~~~~~~~~~~~$if$~~$n$~~is~~even~~or~~$n$~~is ~~odd~~and~~p \equiv 1(mod~4),\\
\pm \sqrt{-1} \zeta_{p^{k-(s-1)t}}^{j} \zeta_{p^t}^{\mathbf{c} \cdot \mathbf{d}} ~~~~ $if$~~ $n$~~is ~~odd~~and~~p \equiv 3(mod~4),
\end{array}
\right.
\end{eqnarray*}
where $j$ and $d$ only depend on $\mathbf{u}$ and $f$.
\end{thm}

\begin{proof}
We show the result by induction.

If $s=1$, the claim is obvious, because $g_{\mathbf{c}}(\mathbf{x})=f(\mathbf{x})$ in this case.

If $s=2$, by taking $g=\sum_{i=1}^{t}p^{i-1}a_i$ and $h=\sum_{i=1}^{k-t}p^{i-1}a_{t+i}$, the claim follows from Theorem \ref{thm3.1.01}.
Since $f$ is gbent if and only if for any $\mathbf{u} \in \Z_p^n$ and $c \in \Z_{p^t}$, there exist some $d \in \Z_{p^t}$ and $j \in \Z_{p^{k-t}}$ such that
\begin{eqnarray*}
\mathcal{H}_{h+cp^{k-2t}g}(\mathbf{u})=\left\{
\begin{array}{l}
\pm \zeta_{p^{k-t}}^{j} \zeta_{p^t}^{cd} ~~~~~~~~~~~~$if$~~$n$~~is~~even~~or~~$n$~~is ~~odd~~and~~p \equiv 1(mod~4),\\
\pm \sqrt{-1} \zeta_{p^{k-t}}^{j} \zeta_{p^t}^{cd} ~~~~ $if$~~ $n$~~is ~~odd~~and~~p \equiv 3(mod~4),
\end{array}
\right.
\end{eqnarray*}
where $j$ and $d$ only depend on $\mathbf{u}$ and $f$. Note that $g_c=h+cp^{k-2t}g$, then it is easy to see
that the claim holds.

Assume the  result is true for some $s$ satisfying $(s+1)t\leq k$, i.e.,  for all $\mathbf{c} \in \Z_{p^t}^{s-1}$, and
\begin{eqnarray}
\label{3.1.04}
g_{\mathbf{c}}(\mathbf{x})=\sum\limits_{i=1}^{k-ts}p^{i-1}a_{t(s-1)+i}(\mathbf{x})+
p^{k-ts}\left(\sum \limits_{i=1}^{s-1}c_i\left(\sum\limits_{j=1}^tp^{j-1}a_{(i-1)t+j}(\mathbf{x})\right)
+\sum\limits_{j=1}^t p^{j-1}a_{k-t+j}(\mathbf{x})\right).
\end{eqnarray}
There exist  some $\mathbf{d} \in \Z_{p^t}^{s-1}$ and $j \in \Z_{p^{k-(s-1)t}}$ such that
\begin{eqnarray*}
\mathcal{H}_{g_{\mathbf{c}}}(\mathbf{u})=\left\{
\begin{array}{l}
\pm \zeta_{p^{k-(s-1)t}}^{j} \zeta_{p^t}^{\mathbf{c} \cdot \mathbf{d}} ~~~~~~~~~~~~$if$~~$n$~~is~~even~~or~~$n$~~is ~~odd~~and~~p \equiv 1(mod~4),\\
\pm \sqrt{-1} \zeta_{p^{k-(s-1)t}}^{j} \zeta_{p^t}^{\mathbf{c} \cdot \mathbf{d}} ~~~~ $if$~~ $n$~~is ~~odd~~and~~p \equiv 3(mod~4),
\end{array}
\right.
\end{eqnarray*}
where $\mathbf{d}$ and $j$ only depend on $\mathbf{u}$ and $f$. We show that the result also holds for $s+1$. 

Note that (\ref{3.1.04}) can be expressed by
\begin{eqnarray*}
g_{\mathbf{c}}(\mathbf{x})&=&\sum\limits_{i=1}^{t}p^{i-1}a_{t(s-1)+i}(\mathbf{x})+ \sum\limits_{i=1}^{k-t(s+1)} p^t\left(p^{i-1}a_{ts+i}(\mathbf{x})+\right. \\
&&\left. p^{k-t(s+1)}\left(\sum \limits_{i=1}^{s-1}c_i\left(\sum\limits_{j=1}^tp^{j-1}a_{(i-1)t+j}(\mathbf{x})\right)
+\sum\limits_{j=1}^t p^{j-1}a_{k-t+j}(\mathbf{x})\right)\right).
\end{eqnarray*}
Let
\begin{eqnarray*}
g(\mathbf{x})=\sum\limits_{i=1}^{t}p^{i-1}a_{t(s-1)+i}(\mathbf{x})
\end{eqnarray*}
and
\begin{eqnarray*}
h(\mathbf{x})=\sum\limits_{i=1}^{k-t(s+1)}p^{i-1}a_{ts+i}(\mathbf{x})+
p^{k-t(s+1)}\left(\sum \limits_{i=1}^{s-1}c_i\left(\sum\limits_{j=1}^tp^{j-1}a_{(i-1)t+j}(\mathbf{x})\right)
+\sum\limits_{j=1}^t p^{j-1}a_{k-t+j}(\mathbf{x})\right).
\end{eqnarray*}
Note that $g_{\mathbf{c}} \in \mathcal{GB}_n^{p^{k-(s-1)t}}$ and $g_{\mathbf{c'}}=h+cp^{k-(s-1)t-2t}g=h+cp^{k-(s+1)t}g \in \mathcal{GB}_n^{p^{k-st}}$,
by Theorem \ref{thm3.1.01},  $g_{\mathbf{c}}$ is gbent if and only if for any $\mathbf{u} \in \Z_p^n$ and $c \in \Z_{p^t}$,
there exist some $d \in \Z_{p^t}$ and $j' \in \Z_{p^{k-st}}$ such that
\begin{eqnarray*}
\mathcal{H}_{g_{\mathbf{c'}}}=\left\{
\begin{array}{l}
\pm \zeta_{p^{k-st}}^{j} \zeta_{p^t}^{\mathbf{c'} \cdot \mathbf{d'}} ~~~~~~~~~~~~$if$~~$n$~~is~~even~~or~~$n$~~is ~~odd~~and~~p \equiv 1(mod~4),\\
\pm \sqrt{-1} \zeta_{p^{k-st}}^{j} \zeta_{p^t}^{\mathbf{c'} \cdot \mathbf{d'}} ~~~~ $if$~~ $n$~~is ~~odd~~and~~p \equiv 3(mod~4),
\end{array}
\right.
\end{eqnarray*}
where $j'$ and $d$ only depend on $\mathbf{u}$ and $g_{\mathbf{c}}$.

This completes the proof.
\end{proof}

\begin{Corollary}
\label{Corollary3.1.02}
Let $f \in \mathcal{GB}_n^{p^k}$ and $f(\mathbf{x})=\sum\limits_{i=1}^{k}p^{i-1}a_i(\mathbf{x})$,
for some $p$-ary Boolean functions $a_i : \Z_p^n \rightarrow \Z_p, 1 \leq i \leq k$.
Let $\mathbf{c}=(c_1,\dots,c_{s-1}) \in \Z_{p}^{s-1}$, and denote 
\begin{eqnarray*}
g_{\mathbf{c}}(\mathbf{x})&=\sum\limits_{i=1}^{k-s}p^{i-1}a_{s-1+i}(\mathbf{x})+p^{k-s}\left(\sum \limits_{i=1}^{s-1}c_ia_i(\mathbf{x})+a_k(\mathbf{x})\right).
\end{eqnarray*}
Then $f$ is  gbent  if and only if $g_{\mathbf{c}}$ is gbent in $\mathcal{GB}_n^{p^{k+1-s}}$  for any $s$ satisfying  $1 \leq s \leq k$, and
there exist  some $\mathbf{d} \in \Z_{p}^{s-1}$ and $j \in \Z_{p^{k-1}}$ such that
\begin{eqnarray*}
\mathcal{H}_{g_{\mathbf{c}}}(\mathbf{u})=\left\{
\begin{array}{l}
\pm \zeta_{p^{k+1-s}}^{j} \zeta_{p}^{\mathbf{c} \cdot \mathbf{d}} ~~~~~~~~~~~~$if$~~$n$~~is~~even~~or~~$n$~~is ~~odd~~and~~p \equiv 1(mod~4),\\
\pm \sqrt{-1} \zeta_{p^{k+1-s}}^{j} \zeta_{p}^{\mathbf{c} \cdot \mathbf{d}} ~~~~ $if$~~ $n$~~is ~~odd~~and~~p \equiv 3(mod~4),
\end{array}
\right.
\end{eqnarray*}
where $j$ and $d$ only depend on $\mathbf{u}$ and $f$.
\end{Corollary}

\begin{proof}
The result follows by setting $t=1$ in Theorem \ref{thm3.1.02}.
\end{proof}

\begin{Lemma}
\label{Lemma3.1.02}
Let $k=lt$ and $f \in \mathcal{GB}_n^{p^k}$ be defined as $f(\mathbf{x})=\sum\limits_{i=1}^{l}p^{(i-1)t}b_i(\mathbf{x})$,
where $b_i \in \mathcal{GB}_n^{p^t}$, $1 \leq i \leq l$. Then for every  $\mathbf{u} \in \Z_p^n$, $$\mathcal{H}_f(\mathbf{u})=
\frac{1}{p^{k-t}} \sum \limits_{\mathbf{c} \in \Z_{p^t}^{l-1}} \mathcal{H}_{b_l+\sum_{j=1}^{l-1}c_jb_j(\mathbf{x})}(\mathbf{u}) \gamma_{\mathbf{c}},$$
where $\gamma_{\mathbf{c}}=\sum \limits_{\mathbf{d} \in \Z_{p^t}^{l-1}} \zeta_{p^t}^{-\mathbf{c}\cdot \mathbf{d}} \zeta_{p^k}^{\sum_{j=1}^{l-1}p^{(j-1)t}d_j}$.
\end{Lemma}

\begin{proof}
By the definition of generalized Walsh-Hadamard transform, we have
\begin{eqnarray*}
p^{\frac{n}{2}} \mathcal{H}_f(\mathbf{u})&=& \sum \limits_{\mathbf{x} \in \Z_p^n} \zeta_p^{-\mathbf{u}\cdot\mathbf{x}} \zeta_{p^k}^{f(\mathbf{x})}
=\sum \limits_{\mathbf{x} \in \Z_p^n} \zeta_p^{-\mathbf{u}\cdot\mathbf{x}} \zeta_{p^t}^{b_l(\mathbf{x})} \prod \limits_{j=1}^{l-1} \zeta_{p^k}^{p^{(j-1)t}b_j(\mathbf{x})}\\
&=&\sum \limits_{\mathbf{x} \in \Z_p^n} \zeta_p^{-\mathbf{u}\cdot\mathbf{x}} \zeta_{p^t}^{b_l(\mathbf{x})} \prod \limits_{j=1}^{l-1}  \frac{1}{p^t}
\left( \sum \limits_{d_j \in \Z_{p^t}}  \left( \sum \limits_{c_j \in \Z_{p^t}} \zeta_{p^t}^{c_j(b_j(\mathbf{x})-d_j)} \right) \zeta_{p^k}^{p^{(j-1)t}d_j} \right)\\
&=&\frac{1}{p^{k-t}} \sum \limits_{\mathbf{x} \in \Z_p^n} \zeta_p^{-\mathbf{u}\cdot\mathbf{x}} \zeta_{p^t}^{b_l(\mathbf{x})} \prod \limits_{j=1}^{l-1}
\left( \sum \limits_{c_j \in \Z_{p^t}}  \left( \sum \limits_{d_j \in \Z_{p^t}} \zeta_{p^t}^{-c_jd_j} \zeta_{p^k}^{p^{(j-1)t}d_j}  \right)\zeta_{p^t}^{c_jb_j(\mathbf{x})} \right)\\
&=&\frac{1}{p^{k-t}} \sum \limits_{\mathbf{x} \in \Z_p^n} \zeta_p^{-\mathbf{u}\cdot\mathbf{x}} \zeta_{p^t}^{b_l(\mathbf{x})}  \sum \limits_{\mathbf{c} \in \Z_{p^t}^{l-1}} \zeta_{p^t}^{\sum_{j=1}^{l-1}c_jb_j(\mathbf{x})}
\prod \limits_{j=1}^{l-1}   \left( \sum \limits_{d_j \in \Z_{p^t}} \zeta_{p^t}^{-c_jd_j} \zeta_{p^k}^{p^{(j-1)t}d_j}  \right)\\
&=&\frac{1}{p^{k-t}} \sum \limits_{\mathbf{c} \in \Z_{p^t}^{l-1}} \sum \limits_{\mathbf{x} \in \Z_p^n} \zeta_p^{-\mathbf{u}\cdot\mathbf{x}}  \zeta_{p^t}^{b_l(\mathbf{x})+\sum_{j=1}^{l-1}c_jb_j(\mathbf{x})}
\prod \limits_{j=1}^{l-1}   \left( \sum \limits_{d_j \in \Z_{p^t}} \zeta_{p^t}^{-c_jd_j} \zeta_{p^k}^{p^{(j-1)t}d_j}  \right)\\
&=&\frac{p^{\frac{n}{2}}}{p^{k-t}} \sum \limits_{\mathbf{c} \in \Z_{p^t}^{l-1}}  \mathcal{H}_{b_l+\sum_{j=1}^{l-1}c_jb_j} \gamma_{\mathbf{c}}.
\end{eqnarray*}
Note that the last equality holds by Lemma \ref{Lemma2.7}.  Then the result follows.
\end{proof}

\begin{thm}
\label{thm3.1.03}
Let $k=lt$ and $f \in \mathcal{GB}_n^{p^k}$ be defined as $f(\mathbf{x})=\sum\limits_{i=1}^{l}p^{(i-1)t}b_i(\mathbf{x})$,
where $b_i \in \mathcal{GB}_n^{p^t}$, $1 \leq i \leq l$.  Then $f$ is gbent if and only if for any $\mathbf{u} \in \Z_p^n$
and $\mathbf{c }\in \Z_{p^t}^{l-1}$, there exists some  $\mathbf{d }\in \Z_{p^t}^{l-1}$ and $j \in \Z_{p^{t}}$ such that

\begin{eqnarray*}
\mathcal{H}_{b_l+\sum_{j=1}^{l-1}c_jb_j}(\mathbf{u})=\left\{
\begin{array}{l}
\pm  \zeta_{p^t}^{j+\mathbf{c} \cdot \mathbf{d}} ~~~~~~~~~~~~$if$~~$n$~~is~~even~~or~~$n$~~is ~~odd~~and~~p \equiv 1(mod~4),\\
\pm \sqrt{-1} \zeta_{p^t}^{j+\mathbf{c} \cdot \mathbf{d}} ~~~~ $if$~~ $n$~~is ~~odd~~and~~p \equiv 3(mod~4),
\end{array}
\right.
\end{eqnarray*}
where $b_l+\sum_{j=1}^{l-1}c_jb_j \in \mathcal{GB}_n^{p^{t}}$ for every $\mathbf{c} \in \Z_{p^t}^{l-1}$, and $\mathbf{d}$ and $j$ only depend on $\mathbf{u}$ and $f$.
\end{thm}

\begin{proof}
If $n$ is even or $n$ is odd and $p\equiv 1~(mod~4)$, then $\mathcal{H}_f{(\textbf{u})}= \pm \zeta_{p^k}^i$, for some $0 \leq i \leq p^k-1$,
by Lemma \ref{Lemma2.3}.
Hence, $\mathcal{H}_f{(\textbf{u})}$ can be expressed as $\mathcal{H}_f{(\textbf{u})}= \pm \zeta_{p^{t}}^j \zeta_{p^k}^{i-jp^{k-t}}$,
where $0 \leq j \leq p^{t}-1$, and let $d=i-jp^{k-t}$, $0 \leq d \leq p^{k-t}-1$. Further, note that $k=lt$, $d$ can be expressed by
$d=\sum_{j=1}^{l-1}p^{(j-1)t}d_j$. We denote $\mathbf{d}=(d_1,\dots, d_{l-1}) \in \Z_{p^t}^{l-1}$.

According to Lemma \ref{Lemma3.1.02}, we have
\begin{eqnarray}
\label{3.1.05}
\mathcal{H}_f{(\textbf{u})}&=& \frac{1}{p^{k-t}} \sum \limits_{\mathbf{c} \in \Z_{p^t}^{l-1}} \mathcal{H}_{b_l+\sum_{j=1}^{l-1}c_jb_j}(\mathbf{u}) \gamma_{\mathbf{c}},
\end{eqnarray}
where $\gamma_{\mathbf{c}}=\sum \limits_{\mathbf{d} \in \Z_{p^t}^{l-1}} \zeta_{p^t}^{-\mathbf{c}\cdot \mathbf{d}} \zeta_{p^k}^{\sum_{j=1}^{l-1}p^{(j-1)t}d_j}$.
By  Lemma \ref{Lemma2.6},  we have
\[\zeta_{p^k}^{d}=\frac{1}{p^{k-t}} \sum \limits_{c\in \Z_{p^t}^{l-1}} \zeta_{p^t}^{\mathbf{c} \cdot \mathbf{d}} \gamma_{\mathbf{c}}.\]
Therefore,
\begin{eqnarray}
\label{3.1.06}
\begin{split}
\mathcal{H}_f{(\textbf{u})}&= \pm \zeta_{p^{t}}^j \zeta_{p^k}^{i-jp^{k-t}}=\pm \zeta_{p^{t}}^j \zeta_{p^k}^{d}\\
&= \pm \frac{1}{p^{k-t}}\zeta_{p^{t}}^j \sum \limits_{c\in \Z_{p^t}^{l-1}} \zeta_{p^t}^{\mathbf{c} \cdot \mathbf{d}} \gamma_{\mathbf{c}}.
\end{split}
\end{eqnarray}
Combining  (\ref{3.1.05}) with (\ref{3.1.06}),  we can get
\begin{eqnarray}
\label{3.1.07}
\begin{split}
\mathcal{H}_f{(\textbf{u})}&=\frac{1}{p^{k-t}} \sum \limits_{\mathbf{c} \in \Z_{p^t}^{l-1}}  \mathcal{H}_{b_l+\sum_{j=1}^{l-1}c_jb_j} \gamma_{\mathbf{c}}\\
&=\pm \frac{1}{p^{k-t}}\zeta_{p^{t}}^j \sum \limits_{c\in \Z_{p^t}^{l-1}} \zeta_{p^t}^{\mathbf{c} \cdot \mathbf{d}} \gamma_{\mathbf{c}}.
\end{split}
\end{eqnarray}
Since $b_l+\sum_{j=1}^{l-1}c_jb_j \in \mathcal{GB}_n^{p^{t}}$, it is absolutely that $\mathcal{H}_{b_l+\sum_{j=1}^{l-1}c_jb_j}(\mathbf{u}) \in \Q(\zeta_{p^{t}})$
in the case of $n$ is even or $n$  is odd and $p\equiv 1~(mod~4)$. We note that $\{1,\zeta_{p^k},\dots, \zeta_{p^k}^{p^{k-t}-1}\}$
is a basis of $\Q(\zeta_{p^k})$ over $\Q(\zeta_{p^{t}})$ by condition (i) of Lemma \ref{Lemma2.5}. So, $\{\gamma_{\mathbf{c}}|c\in \Z_{p^t}^{l-1}\}$
is also a basis of $\Q(\zeta_{p^k})$ over $\Q(\zeta_{p^{t}})$ by  Remark \ref{Remark2.4}.

Based on the above  discussions and (\ref{3.1.07}), we can get
\[\mathcal{H}_{b_l+\sum_{j=1}^{l-1}c_jb_j}(\mathbf{u})= \pm  \zeta_{p^t}^{j+\mathbf{c} \cdot \mathbf{d}}.\]

If $n$ is odd and $p\equiv 3~(mod~4)$, then $\mathcal{H}_f{(\textbf{u})}= \pm \sqrt{-1}\zeta_{p^k}^i$, for some $0 \leq i \leq p^k-1$,
by Lemma \ref{Lemma2.3}. Note that in this case, by  Lemma \ref{Lemma2.2} (ii), we have
\[\mathcal{H}_{b_l+\sum_{j=1}^{l-1}c_jb_j}(\mathbf{u}) \in \Q(\zeta_{4p^{t}}) \setminus \Q(\zeta_{2p^{t}})  \subseteq \Q(\zeta_{p^{t}},\sqrt{-1}).\]
Similar as before, we can get
\[\mathcal{H}_{b_l+\sum_{j=1}^{l-1}c_jb_j}(\mathbf{u})=\pm \sqrt{-1} \zeta_{p^t}^{j+\mathbf{c} \cdot \mathbf{d}}.\]
\end{proof}

\begin{Remark}
Theorem \ref{thm3.1.03} generalizes \cite[Theorem 3.2]{Wang}.
\end{Remark}

\section{$\Z_{p^k}$-bent functions, vectorial bent functions and relative difference sets}
\label{sec4}

We recall that a $\Z_{p^k}$-bent function is a function from $\Z_{p}^n$ to $\Z_{p^k}$, for which
\begin{eqnarray*}
\mathcal{H}_f(a,\mathbf{u})=p^{-\frac{n}{2}} \sum\limits_{\mathbf{x}\in \Z_{p}^n} \zeta_{p^k}^{af(\mathbf{x})} \zeta_p^{-\mathbf{u}\cdot \mathbf{x}}
\end{eqnarray*}
has absolute value $1$ for every $\mathbf{u} \in \Z_p^n$ and nonzero $a \in \Z_{p^k}$. Firstly, we have the following proposition
for $\Z_{p^k}$-bent function.

\begin{Proposition}
\label{Proposition4.01}
A function $f(\mathbf{x})=\sum_{i=0}^{k-1}p^{i} a_i(\mathbf{x}) \in \mathcal{GB}_n^{p^k}$ is $\Z_{p^k}$-bent if and only if
$p^tf(\mathbf{x})=\sum_{i=0}^{k-t-1}p^{t+i} a_i(\mathbf{x})$ is a gbent function for every $0 \leq t \leq k-1$.
\end{Proposition}
\begin{proof}
If $f$ is $\Z_{p^k}$-bent, then $|\mathcal{H}_f^{(p^k)}(p^t,\mathbf{u})|=|\mathcal{H}_{p^t f}^{(p^k)}(\mathbf{u})|=1$ for every $\mathbf{u} \in \Z_p^n$ and $0 \leq t \leq k-1$
by definition.

For the converse, let $a=p^tz$, for $0 \leq t \leq k-1$ and $(z,p)=1$. We have to show that $|\mathcal{H}_f^{(p^k)}(a,\mathbf{u})|=1$ for all $\mathbf{u} \in \Z_p^n$.
Let $f_t(\mathbf{x})=\sum_{i=0}^{k-t-1}p^{i} a_i(\mathbf{x})$. By the assumption, for all $\mathbf{u} \in \Z_p^n$,
\begin{eqnarray*}
\mathcal{H}_f^{(p^k)}(p^t,\mathbf{u})&=&p^{-\frac{n}{2}}\sum\limits_{\mathbf{x}\in \Z_{p}^n} \zeta_{p^k}^{\sum_{i=0}^{k-t-1}p^{t+i}a_i(\mathbf{x})} \zeta_p^{-\mathbf{u}\cdot \mathbf{x}}\\
&=&p^{-\frac{n}{2}}\sum\limits_{\mathbf{x}\in \Z_{p}^n} \zeta_{p^{k-t}}^{\sum_{i=0}^{k-t-1}p^{i}a_i(\mathbf{x})} \zeta_p^{-\mathbf{u}\cdot \mathbf{x}}\\
&=&\mathcal{H}_{f_t}^{(p^{k-t})}(\mathbf{u}),
\end{eqnarray*}
which has absolute value $1$.

Let $\sigma \in Gal(\Q(\zeta_{p^{k-t}}) / \Q)$ and $\sigma(\zeta_{p^{k-t}})=\zeta_{p^{k-t}}^z$. Then we have
\begin{eqnarray}
\label{4.01}
\sigma\left( p^{\frac{n}{2}} \mathcal{H}_{f_t}^{(p^{k-t})}(\mathbf{u})\right)=p^{\frac{n}{2}} \mathcal{H}_{f_t}^{(p^{k-t})}(z,\tilde{z}\mathbf{u}),
\end{eqnarray}
where $z\equiv \tilde{z}~(mod~p)$. Note that $(z,p)=(\tilde{z},p)=1$,  as $\mathbf{u}$ runs through the elements of $\Z_p^n$, so does $\tilde{z}\mathbf{u}$.
It is easy to verify that $\mathcal{H}_f^{(p^k)}(p^tz,\mathbf{u})=\mathcal{H}_{f_t}^{(p^{k-t})}(z,\mathbf{u})$.  Further, by the assumption
$p^tf(\mathbf{x})=\sum_{i=0}^{k-t-1}p^{t+i} a_i(\mathbf{x})$ is  gbent and (\ref{4.01}),  we can get
$|\mathcal{H}_f^{(p^k)}(p^tz,\mathbf{u})|= | \mathcal{H}_{f_t}^{(p^{k-t})}(z,\mathbf{u}) | =1$ for every $\mathbf{u} \in \Z_p^n$.
\end{proof}

\begin{Remark}
\label{Remark4.01}
Note that $p^tf(\mathbf{x})=\sum_{i=0}^{k-t-1}p^{t+i} a_i(\mathbf{x})$ is a gbent function in $\mathcal{GB}_n^{p^k}$
if and only if $f_t(\mathbf{x})=\sum_{i=0}^{k-t-1}p^{i} a_i(\mathbf{x})$ is  gbent  in $\mathcal{GB}_n^{p^{k-t}}$.
For a $\Z_{p^k}$-bent function in $\mathcal{GB}_n^{p^k}$,  by Proposition \ref{Proposition4.01}, we know that
\begin{eqnarray*}
f(\mathbf{x})&=& a_0(\mathbf{x})+pa_1(\mathbf{x})+\cdots+p^{k-1}a_{k-1}(\mathbf{x}) \in \mathcal{GB}_n^{p^k},\\
f_1(\mathbf{x})&=& a_0(\mathbf{x})+pa_1(\mathbf{x})+\cdots+p^{k-2}a_{k-2}(\mathbf{x}) \in \mathcal{GB}_n^{p^{k-1}},\\
&\vdots&\\
f_{k-1}(\mathbf{x})&=& a_0(\mathbf{x})\in \mathcal{GB}_n^{p}
\end{eqnarray*}
are gbent functions ($f_{k-1}$ is $p$-ary bent). Further, by Corollary \ref{Corollary3.1.03}, we can get
$\langle a_0,a_1,\cdots,a_{k-1}\rangle$ is a vector space of bent functions over $\Z_p$, i.e., a vectorial bent function.
\end{Remark}

In addition to applications in cryptography, one motivation for considering bent functions is their relation to objects
in combinatorics. In the following, we point out a relation between relative difference set and $\Z_{p^k}$-bent functions.
Firstly, we recall the definition of a relative difference set. Let $G$ be a group of order $uv$, Let $N$ be a subgroup
of  $G$ of order $v$ and let $R$  be a subset of $G$ of cardinality $k$. Then $R$ is called a $(u,v,k,\lambda)$-relative
difference set of $G$  relative to $N$, if every element $g \in G \setminus N$ can be represented in exactly $\lambda$
ways as difference $r_1-r_2$, $r_1,r_2 \in R$, and no nonzero element of $N$ has such a representation.

Relative difference set can be described by characters as follows.

\begin{Proposition}
[see \cite{Tan}]
\label{Proposition4.02}
Let $G$ be a group order $uv$ and Let $N$ be a subgroup of  $G$ of order $v$. A subset $R$ with  cardinality $k$ of $G$
is a $(u,v,k,\lambda)$-relative difference set of $G$  relative to $N$ if and only if for every character $\chi$ of $G$
\begin{eqnarray*}
|\chi(R)|^2=\left\{
\begin{array}{l}
k^2~~~~~~~~~~~~if~\chi=\chi_0,\\
k-\lambda v ~~~~if~\chi \neq \chi_0, ~but~\chi(g)=1~for~all~g\in N,\\
k~~~~~~~~~~~~~otherwise,
\end{array}
\right.
\end{eqnarray*}
where $\chi(R) =\sum_{x\in R}\chi(x)$.
\end{Proposition}

\begin{thm}
\label{thm4.01}
Let $f:\Z_p^n \rightarrow \Z_{p^k}$ be a $\Z_{p^k}$-bent function. Then the graph $R=\{(\mathbf{x},f(\mathbf{x})):\mathbf{x} \in \Z_p^n\}$
is a $(p^n,p^k,p^n,p^{n-k})$-relative difference set of $G=\Z_p^n \times \Z_{p^k}$  relative to $N=\{0\} \times \Z_{p^k}$.
\end{thm}
\begin{proof}
We denote  the group of characters of $G$ by $\widehat{G}$. It is known that $\widehat{G} \cong \widehat{\Z_{p}^n} \times \widehat{\Z_{p^k}}$, so
$\widehat{G}=\{\chi_{\mathbf{u},a}: \chi_{\mathbf{u},a}(\mathbf{x},y)=\zeta_p^{-\mathbf{u} \cdot \mathbf{x}} \zeta_{p^k}^{ay}, for~all~ (\mathbf{x}, y) \in \Z_{p}^n \times \Z_{p^k} \}$.
For a character $\chi_{\mathbf{u},a}$, we have
\begin{eqnarray*}
|\chi_{\mathbf{u},a}(R)|^2&=&\sum \limits_{(\mathbf{x},f(\mathbf{x})) \in R} \zeta_{p^k}^{af(\mathbf{x})} \zeta_p^{-\mathbf{u} \cdot \mathbf{x}}
\sum \limits_{(\mathbf{y},f(\mathbf{y})) \in R} \overline{ \zeta_{p^k}^{af(\mathbf{y})} \zeta_p^{-\mathbf{u} \cdot \mathbf{y}} } \\
&=&\sum \limits_{\mathbf{x} \in \Z_{p}^n} \zeta_{p^k}^{af(\mathbf{x})} \zeta_p^{-\mathbf{u} \cdot \mathbf{x}}
\sum \limits_{{\mathbf{x} \in \Z_{p}^n} } \overline{ \zeta_{p^k}^{af(\mathbf{y})} \zeta_p^{-\mathbf{u} \cdot \mathbf{y}} } \\
&=&\left\{
\begin{array}{l}
p^{2n}~~~for~\chi_{\mathbf{0},0},\\
0~~~~~~for~\chi_{\mathbf{u},0},\mathbf{u} \neq \mathbf{0},\\
p^n~~~~otherwise.
\end{array}
\right.
\end{eqnarray*}
It is  easy to verify that the  last equality holds for $\chi_{\mathbf{0},0}$ and $\chi_{\mathbf{u},0},\mathbf{u} \neq \mathbf{0}$, and note that $f$
is $\Z_{p^k}$-bent, it also holds for other cases. By proposition  \ref{Proposition4.02}, we conclude the proof.
\end{proof}

We close this section with the following example, which gives a class of $\Z_{p^k}$-bent functions.  This class is defined via spreads.
We start by recalling that a spread of $\Z_{p}^{n}$, $n=2m$, is a family of $p^m+1$ subspaces $U_0,U_1,\dots,U_{p^m}$ of $\Z_{p}^{n}$,
whose pairwise intersection is trivial. The classical example is the regular spread, which for $\Z_p^{2m} \cong \F_{p^m} \times \F_{p^m}$
is represented by the family $S=\bigcup_{s\in \F_{p^m}}\{(x,sx):x\in \F_{p^m}\} \cup \{(0,y):y \in \F_{p^m}\}$.

\begin{Example}
\label{Example4.01}
Let $U_0,U_1, \dots,U_{p^m}$ be the elements of a spread of $\Z_{p}^{n}$, $n=2m$. We first construct a vectorial bent function $F$, and thereafter
a $\Z_{p^k}$-bent function $f$.

Let $\phi:\{1,2,\dots,p^{\frac{n}{2}}\} \rightarrow \Z_p^k$ be a balanced map, thus any $y \in \Z_p^k$ has exactly $p^{n/2-k}$ preimages in the set
$\{1,2,\dots,p^{\frac{n}{2}}\}$. Then the function $F:\Z_p^n \rightarrow \Z_p^k $ given by
\begin{eqnarray*}
F(\mathbf{x})=\left\{
\begin{array}{l}
\phi(s)~~~~~\mathbf{x}\in U_s,~1\leq s \leq p^m,~and~x\neq 0, \\
0~~~~~~~~~~\mathbf{x} \in U_0,
\end{array}
\right.
\end{eqnarray*}
is a vectorial bent function. The proof is similar with the even characteristic case, and we omit it here. The interest reader can refer to \cite[Example 2]{Hodzic1}.

We now construct the $\Z_{p^k}$-bent function. From the balanced map $\phi$,
we obtain in a natural way a balanced map $\bar{\phi}$ from $\{1,2,\dots,p^{\frac{n}{2}}\}$ to $\Z_{p^k}$ defined as
$\bar{\phi}(s)=y_0+py_1+\cdots+p^{k-1}y_{k-1}$ if $\phi(s)=(y_0,y_1,\cdots,y_{k-1})$. Then the function
\begin{eqnarray*}
f(\mathbf{x})=\left\{
\begin{array}{l}
\bar{\phi}(s)~~~~~\mathbf{x}\in U_s,~1\leq s \leq p^m,~and~x\neq 0, \\
0~~~~~~~~~~\mathbf{x} \in U_0,
\end{array}
\right.
\end{eqnarray*}
from $\Z_p^n$ to $\Z_{p^k}$ is $\Z_{p^k}$-bent.  We need  to prove that $|\mathcal{H}_f(c,\mathbf{u})|=1$ for every nonzero $c \in \Z_{p^k}$. Indeed,
\begin{eqnarray*}
\mathcal{H}_f(c,\mathbf{u})&=& p^{-\frac{n}{2}} \sum\limits_{\mathbf{x} \in \Z_p^n}\zeta_{p^k}^{cf(\mathbf{x})}\zeta_p^{-\mathbf{u} \cdot \mathbf{x}}\\
&=&\sum\limits_{s=1}^{p^m}   \sum\limits_{x\in U_s \setminus \{0\}} \zeta_{p^k}^{cf(\mathbf{x})}  \zeta_p^{-\mathbf{u} \cdot \mathbf{x}}
+ \sum\limits_{\mathbf{x}\in U_0}\zeta_p^{-\mathbf{u} \cdot \mathbf{x}}\\
&=&\sum\limits_{s=1}^{p^m}   \zeta_{p^k}^{c\bar{\phi}(s)}  \sum\limits_{x\in U_s}  \zeta_p^{-\mathbf{u} \cdot \mathbf{x}} - \sum\limits_{s=1}^{p^m} \zeta_{p^k}^{c\bar{\phi}(s)}
+ \sum\limits_{\mathbf{x}\in U_0}\zeta_p^{-\mathbf{u} \cdot \mathbf{x}}.
\end{eqnarray*}
Note that $\sum\limits_{s=1}^{p^m} \zeta_{p^k}^{c\bar{\phi}(s)}=0$ for all nonzero $c \in \Z_{p^k}$. Consequently for $\mathbf{u} \neq \mathbf{0}$, we have
\begin{eqnarray*}
\mathcal{H}_f(c,\mathbf{u})&=&\sum\limits_{s=1}^{p^m}   \zeta_{p^k}^{c\bar{\phi}(s)}  \sum\limits_{x\in U_s}  \zeta_p^{-\mathbf{u} \cdot \mathbf{x}}
+ \sum\limits_{\mathbf{x}\in U_0}\zeta_p^{-\mathbf{u} \cdot \mathbf{x}}\\
&=&\left\{
\begin{array}{l}
p^{\frac{n}{2}}~~~~~~~~~~~~~~\mathbf{u} \in U_0^{\perp}, \\
p^{\frac{n}{2}} \zeta_{p^k}^{c\bar{\phi}(\tilde{s})}~~~~~\mathbf{u} \in U_{\tilde{s}}^{\perp}~for~some~1 \leq \tilde{s} \leq p^{\frac{n}{2}}.
\end{array}
\right.
\end{eqnarray*}
Again using that $\sum\limits_{s=1}^{p^m} \zeta_{p^k}^{c\bar{\phi}(s)}=0$, we obtain $\mathcal{H}_f(c,\mathbf{0})=p^{\frac{n}{2}}$, and the proof is completed.
\end{Example}

At last, we emphasize that the property of being $\Z_{p^k}$-bent is much stronger than the property of being vectorial bent.
$\Z_{p^k}$-bent functions are very interesting vectorial bent functions since they correspond two relative difference  sets
with parameters $(p^n,p^k,p^n,p^{n-k})$: first of all, being vectorial bent, they correspond to the relative difference set
$D=\{(\mathbf{x}, a_0(\mathbf{x}),a_1(\mathbf{x}), \dots, a_{k-1}(\mathbf{x})):\mathbf{x} \in \Z_{p}^n\}$ in $\Z_{p}^n \times \Z_{p}^k$,
and secondly, to the relative difference set $R=\{(\mathbf{x},\sum_{i=0}^{k-1}a_i(\mathbf{x})):\mathbf{x} \in \Z_{p}^n\}$
in $\Z_{p}^n \times \Z_{p^k}$. Besides,
it would be interesting to construct more $\Z_{p^k}$-bent  functions.

\section{The dual and Gray map of gbent functions}
\label{sec5}
In this section we firstly attempt to describe the dual $f^{\ast}$ of an arbitrary gbent function in $\mathcal{GB}_n^{p^k}$, then
the generalized Gray map of gbent function is considered.

\subsection{The dual of gbent function}
\label{sec5.1}

\begin{Lemma}
\label{Lemma5.1.1}
Let $f\in \mathcal{GB}_n^{p^k}$ be given as $f(\mathbf{x})=\sum_{i=0}^{k-1}p^ia_i(\mathbf{x})$, and $g_{\mathbf{c}}(\mathbf{x})=a_{k-1}(\mathbf{x})+\sum_{i=0}^{k-2}c_ia_{i}(\mathbf{x})$,
where $\mathbf{c}=(c_0,c_1,\cdots,c_{k-2}) \in \Z_{p}^{k-1}$.  Then

$(i)$.   \[\zeta_{p^k}^{f(\mathbf{x})}= \frac{1}{p^{k-1}} \sum\limits_{\mathbf{c} \in \Z_{p}^{k-1}} \zeta_{p}^{g_{\mathbf{c}}(\mathbf{x})} \gamma_{\mathbf{c}},\]
where $\gamma_{\mathbf{c}}$ defined as in Remark \ref{Remark2.5}.

$(ii)$. For any $\mathbf{u} \in \Z_p^n$, we have
\[\mathcal{H}_f(\mathbf{u})=\frac{1}{p^{k-1}} \sum \limits_{\mathbf{c} \in \Z_p^{k-1}} \mathcal{W}_{g_{\mathbf{c}}}(\mathbf{u}) \gamma_{\mathbf{c}}. \]
\end{Lemma}

\begin{proof}
$(i)$. It is easy to verify that
\begin{eqnarray*}
\zeta_{p^k}^{f(\mathbf{x})}&=&\zeta_{p^k}^{\sum_{i=0}^{k-1}p^ia_i(\mathbf{x})}\\
&=& \zeta_p^{a_{k-1}(\mathbf{x})} \prod_{i=0}^{k-2} \zeta_{p^k}^{p^ia_{i}(\mathbf{x})}\\
&=& \zeta_p^{a_{k-1}(\mathbf{x})} \prod_{i=0}^{k-2} \frac{1}{p} \left(\sum \limits_{d_i \in \Z_p}
\left(\sum\limits_{c_i \in \Z_p} \zeta_p^{(a_i(\mathbf{x})-d_i)c_i}\right)   \zeta_{p^k}^{p^id_i} \right)\\
&=& \frac{1}{p^{k-1}} \zeta_p^{a_{k-1}(\mathbf{x})} \prod_{i=0}^{k-2}  \left( \sum \limits_{c_i \in \Z_p}
\left(\sum\limits_{d_i \in \Z_p} \zeta_p^{-d_ic_i}  \zeta_{p^k}^{p^id_i} \right)  \zeta_p^{c_ia_i(\mathbf{x})} \right)\\
&=& \frac{1}{p^{k-1}} \zeta_p^{a_{k-1}(\mathbf{x})}  \sum \limits_{\mathbf{c} \in \Z_p^{k-1}} \zeta_p^{\sum_{i=0}^{k-2} c_ia_i(\mathbf{x})}
\prod_{i=0}^{k-2} \left(\sum\limits_{d_i \in \Z_p} \zeta_p^{-d_ic_i}  \zeta_{p^k}^{p^id_i} \right)   \\
&=& \frac{1}{p^{k-1}}  \sum \limits_{\mathbf{c} \in \Z_p^{k-1}} \zeta_p^{ a_{k-1}(\mathbf{x}) + \sum_{i=0}^{k-2} c_ia_i(\mathbf{x})}
\prod_{i=0}^{k-2} \left(\sum\limits_{d_i \in \Z_p} \zeta_p^{-d_ic_i}  \zeta_{p^k}^{p^id_i} \right)   \\
&=& \frac{1}{p^{k-1}}  \sum \limits_{\mathbf{c} \in \Z_p^{k-1}} \zeta_p^{g_{\mathbf{c}}(\mathbf{x})} \gamma_{\mathbf{c}},
\end{eqnarray*}
where the last equality holds from Lemma \ref{Lemma2.7}.

$(ii)$. By the definition, we have
\begin{eqnarray*}
\mathcal{H}_f(\mathbf{u})&=& p^{-\frac{n}{2}} \sum\limits_{\mathbf{x} \in \Z_p^n} \zeta_{p^k}^{f(\mathbf{x})} \zeta_p^{-\mathbf{u} \cdot \mathbf{x}}.
\end{eqnarray*}
Using $(i)$, we get
\begin{eqnarray*}
\mathcal{H}_f(\mathbf{u})&=& \frac{p^{-\frac{n}{2}}}{p^{k-1}}  \sum\limits_{\mathbf{x} \in \Z_p^n}  \sum\limits_{\mathbf{c} \in \Z_{p}^{k-1}}
\zeta_{p}^{g_{\mathbf{c}}(\mathbf{x})}  \zeta_p^{-\mathbf{u} \cdot \mathbf{x}} \gamma_{\mathbf{c}}\\
&=& \frac{p^{-\frac{n}{2}}}{p^{k-1}}  \sum\limits_{\mathbf{c} \in \Z_{p}^{k-1}} \left(\sum\limits_{\mathbf{x} \in \Z_p^n}
\zeta_{p}^{g_{\mathbf{c}}(\mathbf{x})}  \zeta_p^{-\mathbf{u} \cdot \mathbf{x}} \right) \gamma_{\mathbf{c}}\\
&=& \frac{1}{p^{k-1}} \sum \limits_{\mathbf{c} \in \Z_p^{k-1}} \mathcal{W}_{g_{\mathbf{c}}}(\mathbf{u}) \gamma_{\mathbf{c}}.
\end{eqnarray*}

This completes the proof.
\end{proof}

From Lemma \ref{Lemma2.3}, we know that
for gbent function $f \in \mathcal{GB}_n^{p^k}$,  there exists a function $f^{\ast}: \Z_p^{n}\rightarrow \Z_{p^k}$ such that
\begin{eqnarray*}
 \mathcal{H}_f(\mathbf{u})=\left\{
\begin{array}{l}
\pm \zeta_{p^k}^{f^{\ast}(\mathbf{u})}  ~~~~~~~~~~~~$if$~~$n$~~is~~even~~or~~$n$~~is ~~odd~~and~~p \equiv 1(mod~4),\\
\pm \sqrt{-1} \zeta_{p^k}^{f^{\ast}(\mathbf{u})} ~~~~ $if$~~ $n$~~is ~~odd~~and~~p \equiv 3(mod~4).
\end{array}
\right.
\end{eqnarray*}

In \cite{Cesmelioglu1},  $f^{\ast}$ is called the dual of gbent function $f$. 
We emphasis that when $f$ is non-weakly regular,
$f^{\ast}$ may not gbent function. Then, in the following, we will describe the dual $f^{\ast}$ of a gbent function
$f \in \mathcal{GB}_n^{p^k}$ via the dual of the component functions of $f$.

\begin{thm}
\label{thm5.1.1}
Let $f\in \mathcal{GB}_n^{p^k}$ be a gbent function  given as $f(\mathbf{x})=\sum_{i=0}^{k-1}p^ia_i(\mathbf{x})$,
for some $p$-ary Boolean functions $a_i$, $0 \leq i \leq k-1$, with component functions $g_{\mathbf{c}}(\mathbf{x})
=a_{k-1}(\mathbf{x})+\sum_{i=0}^{k-2}c_ia_{i}(\mathbf{x})$, where $\mathbf{c}=(c_0,c_1,\cdots,c_{k-2}) \in \Z_{p}^{k-1}$.
Then the dual $f^{\ast}\in \mathcal{GB}_n^{p^k}$ of $f$ is given as  $f^{\ast}(\mathbf{x})=\sum_{i=0}^{k-1}p^ib_i(\mathbf{x})$,
where $b_{k-1}(\mathbf{x})=a_{k-1}^{\ast}(\mathbf{x})$,  $b_j(\mathbf{x})=\left(a_{k-1}(\mathbf{x}) +a_j(\mathbf{x})\right)^{\ast}-a_{k-1}^{\ast}(\mathbf{x})$, $0 \leq j \leq k-2$.
\end{thm}

\begin{proof}
From the $(ii)$ of Lemma \ref{Lemma5.1.1} and Lemma \ref{Lemma2.3} , we have
\begin{eqnarray}
\begin{split}
\label{5.1.01}
\mathcal{H}_f(\mathbf{u})&=\frac{1}{p^{k-1}} \sum \limits_{\mathbf{c} \in \Z_p^{k-1}} \mathcal{W}_{g_{\mathbf{c}}}(\mathbf{u}) \gamma_{\mathbf{c}}\\
&=\frac{1}{p^{k-1}} \sum \limits_{\mathbf{c} \in \Z_p^{k-1}} \alpha_{\mathbf{c}} \zeta_p^{g_{\mathbf{c}}^{\ast}(\mathbf{u})} \gamma_{\mathbf{c}}\\
&= \beta \zeta_{p^k}^{f^{\ast}(\mathbf{u})}.
\end{split}
\end{eqnarray}
By the  proof  of Theorem \ref{thm3.1.03} (set $t=1$), we know that $\alpha_{\mathbf{c}} =\beta \in \{\pm 1, \pm \sqrt{-1}\}$ for all $\mathbf{c} \in \Z_{p}^{k-1}$.

Suppose that $f^{\ast}(\mathbf{x})=\sum_{i=0}^{k-1}p^ib_i(\mathbf{x})$ and denote the component functions of $f^{\ast}$ by
$h_{\mathbf{c}}(\mathbf{x})=b_{k-1}(\mathbf{x})+\sum_{i=0}^{k-2}c_ib_{i}(\mathbf{x})$ for every $\mathbf{c} \in \Z_{p}^{k-1}$.
By  Lemma \ref{Lemma5.1.1} (i), we have
\begin{eqnarray}
\label{5.1.02}
\zeta_{p^k}^{f^{\ast}(\mathbf{u})}= \frac{1}{p^{k-1}} \sum\limits_{\mathbf{c} \in \Z_{p}^{k-1}} \zeta_{p}^{h_{\mathbf{c}}(\mathbf{u})} \gamma_{\mathbf{c}}
\end{eqnarray}
Combining (\ref{5.1.01}) with (\ref{5.1.02}),  we  can  get
\begin{eqnarray*}
\sum\limits_{\mathbf{c} \in \Z_{p}^{k-1}} \zeta_{p}^{g_{\mathbf{c}}^{\ast}(\mathbf{u})} \gamma_{\mathbf{c}} =
\sum\limits_{\mathbf{c} \in \Z_{p}^{k-1}} \zeta_{p}^{h_{\mathbf{c}}(\mathbf{u})} \gamma_{\mathbf{c}}.
\end{eqnarray*}
Note that $\{\gamma_{\mathbf{c}} | \mathbf{c} \in \Z_p^{k-1}\}$ is linear independently over $\Q(\zeta_p)$ by Remark \ref{Remark2.2}, thus we have
\begin{eqnarray*}
\zeta_{p}^{g_{\mathbf{c}}^{\ast}(\mathbf{u})}=\zeta_{p}^{h_{\mathbf{c}}(\mathbf{u})}
\end{eqnarray*}
for all $\mathbf{c} \in \Z_{p}^{k-1}$. Further,  $g_{\mathbf{c}}^{\ast}(\mathbf{u})=h_{\mathbf{c}}(\mathbf{u})$, for all $\mathbf{c} \in \Z_{p}^{k-1}$.
For simplicity, in the following we also use $\sum_{i=0}^{k-2} p^ic_i$ represents $\mathbf{c}=(c_0,c_1,\cdots,c_{k-2}) \in \Z_p^{k-1}$.
Finally, $b_{k-1}=h_0(\mathbf{u})=g_{0}^{\ast}(\mathbf{u})=a_{k-1}^{\ast} $ and with $g_{p^j}^{\ast}=h_{p^j}=b_{k-1}+b_j$ and $g_{p^j}=a_{k-1}+a_j$,
$0 \leq j \leq k-2$, we get
\begin{eqnarray*}
b_j=(a_{k-1}+a_j)^{\ast}-a_{k-1}^{\ast}, ~~0 \leq j \leq k-2.
\end{eqnarray*}
\end{proof}

\subsection{The generalized Gray map  of gbent function}
\label{sec5.2}

In this subsection, we firstly define the $s$-$plateaued$ function for $p$-ary Boolean functions, and then show t Gray image of
any gbent function in $\mathcal{GB}_n^{p^k}$  is a $(k-1)$-plateaued function.

\begin{Definition}
\label{Definition5.2.1}
For a $p$-ary Boolean function $f$, then $f$ is  called $s$-$plateaued$ if there exist some $a \in \Z_p$, such that
\begin{eqnarray*}
 \mathcal{W}_f(\mathbf{u})=\left\{
\begin{array}{l}
0 ~or~\pm \zeta_p^a p^{\frac{s}{2}} ~~~~~~~~~~~~$if$~~$n$~~is~~even~~or~~$n$~~is ~~odd~~and~~p \equiv 1(mod~4),\\
0 ~or~\pm \sqrt{-1}\zeta_p^a p^{\frac{s}{2}} ~~~~ $if$~~ $n$~~is ~~odd~~and~~p \equiv 3(mod~4).
\end{array}
\right.
\end{eqnarray*}
\end{Definition}

We recall the definition of Gray mapping.
\begin{Definition}
[see \cite{Carlet1}]
\label{Definition5.2.1}
The Gray map is the mapping $G$ from $\Z_4$ to $\Z_2^2$ defined by
\[G(0)=(0,0); G(1)=(0,1); G(2)=(1,1); G(3)=(1,0).\]
\end{Definition}
$G$ can be extended to a mapping from $\Z_4^n$ to $\Z_2^{2n}$ coordinate-wisely. It was also generalized to $\Z_{2^k}$ by Carlet in \cite{Carlet1}.
Recently, the Gary map was generalized to $\Z_{p^k}$ by  Heng et al in \cite{Heng}.  Let $f \in \mathcal{GB}_n^{p^k}$ and be give as
$f(\mathbf{x})=\sum_{i=0}^{k-1}p^ia_i(\mathbf{x})$, for any $\mathbf{x} \in \Z_p^n$.  The  generalized Gray map $G:\mathcal{GB}_n^{p^k}
\rightarrow \mathcal{GB}_{n+k-1}^p$ is defined by
\begin{eqnarray*}
&G(f): \Z_p^n \times \Z_p^{k-1} \rightarrow \Z_p\\
&(\mathbf{x},y_0,\dots,y_{k-2}) \mapsto  a_{k-1}(\mathbf{x})+\sum\limits_{i=0}^{k-2}a_i(\mathbf{x})y_i.
\end{eqnarray*}
It is known that the reverse image of the Hamming distance by the generalized Gray map is a translation-invariant distance.

We now can show that the generalized Gray map of a gbent function in $\mathcal{GB}_n^{p^k}$ is a $(k-1)$-plateaued function.
\begin{thm}
\label{thm5.2.1}
Let $f\in \mathcal{GB}_n^{p^k}$ be a gbent function  given as $f(\mathbf{x})=\sum_{i=0}^{k-1}p^ia_i(\mathbf{x})$. Then $G(f)$
is a $(k-1)$-plateaued function in $\mathcal{GB}_{n+k-1}^p$. For every $(\mathbf{u},\mathbf{z}_r) \in  \Z_p^n \times \Z_p^{k-1}$,
where $\mathbf{z}_r=(z_0,z_1,\dots,z_{k-2}) \in \Z_p^{k-1}$ and $r=\sum_{i=0}^{k-2}p^iz_i$. We have
\begin{eqnarray*}
 \mathcal{W}_{G(f)}(\mathbf{u},\mathbf{z}_r)=\left\{
\begin{array}{l}
0 ~or~\pm  \zeta_p^a p^{\frac{k-1}{2}} ~~~~~~~~~~~~$if$~~$n$~~is~~even~~or~~$n$~~is ~~odd~~and~~p \equiv 1(mod~4),\\
0 ~or~\pm \sqrt{-1} \zeta_p^a  p^{\frac{k-1}{2}} ~~~~ $if$~~ $n$~~is ~~odd~~and~~p \equiv 3(mod~4).
\end{array}
\right.
\end{eqnarray*}
\end{thm}

\begin{proof}
When $n$ is even or $n$ is odd and $p\equiv 1~(mod~4)$.  For every $(\mathbf{u},\mathbf{z}_r) \in  \Z_p^n \times \Z_p^{k-1}$, we have
\begin{eqnarray*}
\mathcal{W}_{G(f)}(\mathbf{u},\mathbf{z}_r)&=& p^{- \frac{n+k-1}{2}} \sum\limits_{(\mathbf{x},\mathbf{z}_j) \in \Z_p^n \times \Z_p^{k-1}}
\zeta_p^{- \mathbf{u} \cdot \mathbf{x} - \mathbf{z}_r \cdot \mathbf{z}_j   + G(f)(\mathbf{x},\mathbf{z}_j)  }  \\
&=& p^{- \frac{n+k-1}{2}} \sum\limits_{\mathbf{z}_j \in \Z_p^{k-1}}  \zeta_p^{- \mathbf{z}_r \cdot \mathbf{z}_j  }
\sum\limits_{\mathbf{x} \in \Z_p^{n}}  \zeta_p^{- \mathbf{u} \cdot \mathbf{x}  + a_{k-1}(\mathbf{x}) + \mathbf{z}_j \cdot A(\mathbf{x}) }  \\
&=& p^{- \frac{k-1}{2}} \sum\limits_{\mathbf{z}_j \in \Z_p^{k-1}}  \zeta_p^{- \mathbf{z}_r \cdot \mathbf{z}_j  }
\mathcal{W}_{a_{k-1}+  \mathbf{z}_j \cdot A } (\mathbf{u}) \\
&=& p^{- \frac{k-1}{2}} \overline{H_{p^{k-1}}^{(r)}} (\pm \zeta_p^a H_{p^{k-1}}^{(j)})^{\mathrm{T}} \\
&=&\left\{
\begin{array}{l}
\pm \zeta_p^a p^{\frac{k-1}{2}} ~~~~if~r=j,\\
0~~~~~~~~~~~~~~~~~if~r\neq j,
\end{array}
\right.
\end{eqnarray*}
where $H_{p^{k-1}}=\left(\zeta_p^{ij}\right)_{0 \leq i,j \leq p-1}^{\bigotimes {(k-1)}}$ is a generalized Hadamard matrix of order $p^{k-1}$,  $A(\mathbf{x})=(a_0(\mathbf{x}),a_1(\mathbf{x}),\dots,a_{k-2}(\mathbf{x}))$
and $A=(a_0,a_1,\dots,a_{k-2})$. Note that the penultimate equal sign holds due to Corollary \ref{Corollary3.1.03}.

When $n$ is odd and $p\equiv 3~(mod~4)$, it can be proved similarly as above, we omit the details here.
\end{proof}

\section{Gbent functions from $\Z_{p^l}^n$ to $\Z_{p^k}$}
\label{sec3.2}
In this part, we consider generalized Boolean functions from  $\Z_{p^l}^n$ to $\Z_{p^k}$, where $l, ~k$ are positive integers.
The generalized Walsh-Hadamard transform of such a function $f$ can be defined by
\begin{eqnarray*}
\mathcal{H}_f(\mathbf{u})=p^{ -\frac{ln}{2} } \sum \limits_{\mathbf{x} \in \Z_{p^l}^n} \zeta_{p^l}^{-\mathbf{u}\cdot \mathbf{x}} \zeta_{p^k}^{f(\mathbf{x})}.
\end{eqnarray*}
$f$ is called a gbent function if $|\mathcal{H}_f(\mathbf{u})|=1$ for all $\mathbf{u}$.

If $l \geq k$, $$\mathcal{H}_f(\mathbf{u})=p^{ -\frac{ln}{2} } \sum_{\mathbf{x} \in \Z_{p^l}^n} \zeta_{p^l}^{-\mathbf{u}\cdot \mathbf{x}} \zeta_{p^k}^{f(\mathbf{x})}=p^{ -\frac{ln}{2} } \sum_{\mathbf{x} \in \Z_{p^l}^n} \zeta_{p^l}^{-\mathbf{u}\cdot \mathbf{x}} \zeta_{p^l}^{p^{l-k}f(\mathbf{x})},$$ then $f$ is gbent
from $\Z_{p^l}^n $ to $ \Z_{p^k}$  if and only if $p^{l-k}f(\mathbf{x})$ is gbent from $\Z_{p^l}^n $ to $ \Z_{p^l}$. Such gbent functions were introduced in \cite{Kumar} and were studied  a lot in the literature.  In the following, we only consider   generalized bent functions
from $\Z_{p^l}^n $ to $ \Z_{p^k}$ with $l <k$.

\begin{Lemma}
\label{Lemma3.2.01}
Let $l<k$ and $f(\mathbf{x})=\sum_{i=0}^{k-1}f_i(\mathbf{x})p^{i}$, where $f : \Z_{p^l}^n \rightarrow \Z_{p^k}$ and $f_i: \Z_{p^l}^n \rightarrow \Z_{p}$.
Then
\begin{equation*}
\mathcal{H}_f(\textbf{u})=\frac{1}{p^{k-l}} \sum \limits_{\mathbf{c} \in \Z_p^{k-l}}     \mathcal{H}_{ \sum_{i=k-l}^{k-1}f_ip^{i-(k-l)}+p^{l-1} \left({\sum_{i=0}^{k-l-1} c_if_i}\right)} (\mathbf{u}) \gamma_{\mathbf{c}},
\end{equation*}
where $\gamma_{\textbf{c}}=\sum_{\textbf{d}\in \Z_{p}^{k-l}} \zeta_p^{- \textbf{c}\cdot \textbf{d}} \zeta_{p^k}^{\sum_{i=0}^{k-l-1}d_ip^i}$.
\end{Lemma}

\begin{proof}
According to the definition of $\mathcal{H}_f({\textbf{u}})$, we have
\begin{eqnarray*}
p^{\frac{nl}{2}}\mathcal{H}_f({\textbf{u}})&=&\sum \limits_{\textbf{x}\in \Z_{p^l}^{n}} \zeta_{p^l}^{- \textbf{u}\cdot \textbf{x}} \zeta_{p^k}^{f(\mathbf{x})}\\
&=&\sum \limits_{\textbf{x}\in \Z_{p^l}^{n}} \zeta_{p^l}^{- \textbf{u}\cdot \textbf{x}} \zeta_{p^k}^{\sum_{i=0}^{k-1}f_i(\mathbf{x})p^{i}}\\
&=&\sum \limits_{\textbf{x}\in \Z_{p^l}^{n}} \zeta_{p^l}^{- \textbf{u}\cdot \textbf{x}} \zeta_{p^k}^{\sum_{i=k-l}^{k-1}f_i(x)p^{i}}    \prod \limits_{i=0}^{k-l-1} \zeta_{p^{k}}^{f_i(\mathbf{x})p^i}\\
&=&\sum \limits_{\textbf{x}\in \Z_{p^l}^{n}} \zeta_{p^l}^{- \textbf{u}\cdot \textbf{x} + \sum_{i=k-l}^{k-1}f_i(\mathbf{x})p^{i-(k-l)}}  \prod \limits_{i=0}^{k-l-1} \zeta_{p^{k}}^{f_i(\mathbf{x})p^i}\\
&=&\sum \limits_{\textbf{x}\in \Z_{p^l}^{n}} \zeta_{p^l}^{- \textbf{u}\cdot \textbf{x} + \sum_{i=k-l}^{k-1}f_i(\mathbf{x})p^{i-(k-l)}} \prod \limits_{i=0}^{k-l-1}
\left(\frac{1}{p}\sum \limits_{d_i \in \Z_p} \left(\sum \limits_{c_i \in \Z_p} \zeta_p^{(f_i(\mathbf{x})-d_i)c_i} \right) \zeta_{p^{k}}^{p^id_i}\right) \\
&=&\frac{1}{p^{k-l}}\sum \limits_{\textbf{x}\in \Z_{p^l}^{n}} \zeta_{p^l}^{- \textbf{u}\cdot \textbf{x} + \sum_{i=k-l}^{k-1}f_i(\mathbf{x})p^{i-(k-l)}} \prod \limits_{i=0}^{k-l-1}
\left(\sum \limits_{c_i \in \Z_p} \left(\sum \limits_{d_i \in \Z_p} \zeta_p^{-c_i d_i} \zeta_{p^{k}}^{p^id_i} \right) \zeta_p^{c_if_i(\mathbf{x})}\right) \\
&=&\frac{1}{p^{k-l}}\sum \limits_{\textbf{x}\in \Z_{p^l}^{n}} \zeta_{p^l}^{- \textbf{u}\cdot \textbf{x} + \sum_{i=k-l}^{k-1}f_i(\mathbf{x})p^{i-(k-l)}} \sum \limits_{\mathbf{c} \in \Z_p^{k-l}}
\zeta_p^{\sum_{i=0}^{k-l-1} c_if_i(\mathbf{x})} \prod \limits_{i=0}^{k-l-1} \left(\sum \limits_{d_i \in \Z_p} \zeta_p^{-c_i d_i} \zeta_{p^{k}}^{p^id_i} \right) \\
&=&\frac{1}{p^{k-l}} \sum \limits_{\mathbf{c} \in \Z_p^{k-l}}  \sum \limits_{\textbf{x}\in \Z_{p^l}^{n}} \zeta_{p^l}^{- \textbf{u}\cdot \textbf{x} + \sum_{i=k-l}^{k-1}f_i(\mathbf{x})p^{i-(k-l)}+p^{l-1}
\left({\sum_{i=0}^{k-l-1} c_if_i(\mathbf{x})}\right)}\prod \limits_{i=0}^{k-l-1} \left(\sum \limits_{d_i \in \Z_p} \zeta_p^{-c_i d_i} \zeta_{p^{k}}^{p^id_i} \right) \\
&=&\frac{p^{\frac{nt}{2}}}{p^{k-l}} \sum \limits_{\mathbf{c} \in \Z_p^{k-l}}     \mathcal{H}_{ \sum_{i=k-l}^{k-1}f_ip^{i-(k-l)}+p^{l-1} \left({\sum_{i=0}^{k-l-1} c_if_i}\right)} (\mathbf{u}) \gamma_{\mathbf{c}}.
\end{eqnarray*}
The last equality holds from Remark \ref{Remark2.5}. 
This completes the proof.
\end{proof}

\begin{thm}
\label{thm3.2.01}
Let $l<k$ and $f(\mathbf{x})=\sum_{i=0}^{k-1}f_i(\mathbf{x})p^{i}$, where $f : \Z_{p^l}^n \rightarrow \Z_{p^k}$ and $f_i: \Z_{p^l}^n \rightarrow \Z_{p}$.
Then
$f$ is gbent if and only if for any $\mathbf{u} \in \Z_p^n$ and $\mathbf{c }\in \Z_{p}^{k-l}$, there exists some $\mathbf{d }\in \Z_{p}^{k-l}$ and $j \in \Z_{p^{l}}$ such that
\begin{equation*}
\begin{split}
&\mathcal{H}_{ \sum_{i=k-l}^{k-1}f_ip^{i-(k-l)}+p^{l-1} \left({\sum_{i=0}^{k-l-1} c_if_i}\right) }(\mathbf{u})\\
&=\left\{
\begin{array}{ll}
\pm  \zeta_{p^l}^j\zeta_{p}^{\mathbf{c} \cdot \mathbf{d}} &\text{if}~n~\text{is~even~or}~n~\text{is~odd~and~}p \equiv 1(mod~4),\\
\pm \sqrt{-1} \zeta_{p^l}^j \zeta_{p}^{\mathbf{c} \cdot \mathbf{d}}&
\text{if}~n~\text{is~odd~and}~p \equiv 1(mod~4),
\end{array}
\right.
\end{split}
\end{equation*}
where $ \sum_{i=k-l}^{k-1}f_ip^{i-(k-l)}+p^{l-1} \left({\sum_{i=0}^{k-l-1} c_if_i}\right) \in \mathcal{GB}_n^{p^{l}}$ for every $\mathbf{c} \in \Z_{p^t}^{l-1}$,
and $\mathbf{d}$ and $j$ only depend on $\mathbf{u}$ and $f$.
\end{thm}

\begin{proof}
If $n$ is even or $n$ is odd and $p\equiv 1~(mod~4)$, then $\mathcal{H}_f{(\textbf{u})}= \pm \zeta_{p^k}^i$, for some $0 \leq i \leq p^k-1$,
by Lemma \ref{Lemma2.3}.
Hence, $\mathcal{H}_f{(\textbf{u})}$ can be expressed as $\mathcal{H}_f{(\textbf{u})}= \pm \zeta_{p^{l}}^j \zeta_{p^k}^{i-jp^{k-l}}$,
where $0 \leq j \leq p^{l}-1$, and let $d=i-jp^{k-l}$, $0 \leq d \leq p^{k-l}-1$. Further, $d$ can be expressed by
$d=\sum_{j=0}^{k-l-1}p^{j}d_j$. We denote $\mathbf{d}=(d_0,\cdots, d_{k-l-1}) \in \Z_{p}^{k-l}$.

According to Lemma \ref{Lemma3.2.01}, we have
\begin{eqnarray}
\label{3.2.01}
\mathcal{H}_f(\textbf{u})=\frac{1}{p^{k-l}} \sum \limits_{\mathbf{c} \in \Z_p^{k-l}}  \mathcal{H}_{ \sum_{i=k-l}^{k-1}f_ip^{i-(k-l)}+p^{l-1} \left({\sum_{i=0}^{k-l-1} c_if_i}\right)} (\mathbf{u}) \gamma_{\mathbf{c}},
\end{eqnarray}
where $\gamma_{\textbf{c}}=\sum_{\textbf{d}\in \Z_{p}^{k-l}} \zeta_p^{- \textbf{c}\cdot \textbf{d}} \zeta_{p^k}^{\sum_{i=0}^{k-l-1}d_ip^i}$.
By  Lemma \ref{Lemma2.6},  we have
\[\zeta_{p^k}^{d}=\frac{1}{p^{k-l}} \sum \limits_{c\in \Z_{p}^{k-l}} \zeta_{p}^{\mathbf{c} \cdot \mathbf{d}} \gamma_{\mathbf{c}}.\]
Therefore,
\begin{eqnarray}
\label{3.2.02}
\begin{split}
\mathcal{H}_f{(\textbf{u})}&= \pm \zeta_{p^{l}}^j \zeta_{p^k}^{i-jp^{k-l}}=\pm \zeta_{p^{l}}^j \zeta_{p^k}^{d}\\
&= \pm \frac{1}{p^{k-l}}\zeta_{p^{l}}^j \sum \limits_{c\in \Z_{p}^{k-l}} \zeta_{p}^{\mathbf{c} \cdot \mathbf{d}} \gamma_{\mathbf{c}}.
\end{split}
\end{eqnarray}
Combining  (\ref{3.2.01}) with (\ref{3.2.02}),  we can get
\begin{eqnarray}
\label{3.2.03}
\begin{split}
\mathcal{H}_f(\textbf{u})&=\frac{1}{p^{k-l}} \sum \limits_{\mathbf{c} \in \Z_p^{k-l}}  \mathcal{H}_{ \sum_{i=k-l}^{k-1}f_ip^{i-(k-l)}+p^{l-1} \left({\sum_{i=0}^{k-l-1} c_if_i}\right)} (\mathbf{u}) \gamma_{\mathbf{c}}\\
&=\pm \frac{1}{p^{k-l}}\zeta_{p^{l}}^j \sum \limits_{c\in \Z_{p}^{k-l}} \zeta_{p}^{\mathbf{c} \cdot \mathbf{d}} \gamma_{\mathbf{c}}.
\end{split}
\end{eqnarray}
Since $ \sum_{i=k-l}^{k-1}f_ip^{i-(k-l)}+p^{l-1} \left({\sum_{i=0}^{k-l-1} c_if_i}\right) \in \mathcal{GB}_n^{p^{l}}$, it is absolutely that
$\mathcal{H}_{\sum_{i=k-l}^{k-1}f_ip^{i-(k-l)}+p^{l-1} \left({\sum_{i=0}^{k-l-1} c_if_i}\right)}(\mathbf{u}) \in \Q(\zeta_{p^{l}})$
in the case of $n$ is even or $n$  is odd and $p\equiv 1~(mod~4)$. We note that $\{1,\zeta_{p^k},\cdots, \zeta_{p^k}^{p^{k-l}-1}\}$
is a basis of $\Q(\zeta_{p^k})$ over $\Q(\zeta_{p^{l}})$ by  Lemma \ref{Lemma2.5} (i). So, $\{\gamma_{\mathbf{c}}|c\in \Z_{p}^{k-l}\}$
is also a basis of $\Q(\zeta_{p^k})$ over $\Q(\zeta_{p^{l}})$ by  Remark \ref{Remark2.4}.

Based on the above mentioned discussion and (\ref{3.2.03}), we can get
\[\mathcal{H}_{\sum_{i=k-l}^{k-1}f_ip^{i-(k-l)}+p^{l-1} \left({\sum_{i=0}^{k-l-1} c_if_i}\right)}(\mathbf{u})= \pm  \zeta_{p^l}^j \zeta_{p}^{\mathbf{c} \cdot \mathbf{d}}.\]

If $n$ is odd and $p\equiv 3~(mod~4)$, then $\mathcal{H}_f{(\textbf{u})}= \pm \sqrt{-1}\zeta_{p^k}^i$, for some $0 \leq i \leq p^k-1$,
by Lemma \ref{Lemma2.3}. Note that in this case, by the Lemma \ref{Lemma2.2} (ii), we have
\[\mathcal{H}_{\sum_{i=k-l}^{k-1}f_ip^{i-(k-l)}+p^{l-1} \left({\sum_{i=0}^{k-l-1} c_if_i}\right)}(\mathbf{u}) \in \Q(\zeta_{4p^{l}}) \setminus \Q(\zeta_{2p^{l}})  \subseteq \Q(\zeta_{p^{l}},\sqrt{-1}).\]
Similar as before,  we can get
\[\mathcal{H}_{\sum_{i=k-l}^{k-1}f_ip^{i-(k-l)}+p^{l-1} \left({\sum_{i=0}^{k-l-1} c_if_i}\right)}(\mathbf{u})= \pm \sqrt{-1} \zeta_{p^l}^j \zeta_{p}^{\mathbf{c} \cdot \mathbf{d}}.\]

This completes the proof.
\end{proof}

\begin{Remark}
Theorem \ref{thm3.2.01} coincides with \cite[Theorem 3.2]{Wang} by setting $l=1$.
\end{Remark}

\section{Conclusions}
In this paper, we further investigate properties of generalized bent Boolean functions from $\Z_{p}^n$ to $\Z_{p^k}$. For various kinds of representations,  sufficient and necessary conditions for bent-ness of such functions are given in terms of their various kinds of component functions. Furthermore, $\Z_{p^k}$-bent functions and their relations  to relative difference sets are studied. It turns out that $\Z_{p^k}$-bent functions correspond
  to a class of vectorial bent functions, and  the property of being $\Z_{p^k}$-bent is much stronger then the standard bent-ness.  The dual and the generalized Gray  image of gbent function are also discussed. As a further generalization, we also define and give characterizations of gbent functions from $\Z_{p^l}^n$ to $\Z_{p^k}$ for a positive integer $l$ with $l<k$.


\end{document}